\documentclass[11pt,letterpaper]{amsart}

\usepackage{xcolor}
\usepackage[english]{babel}
\usepackage[T1]{fontenc}
\usepackage[utf8]{inputenc}

\usepackage{amsmath,amssymb,amsthm,mathrsfs}
\usepackage{tikz-cd}
\usepackage{graphicx}
\usepackage{xypic}
\usepackage{hyperref}

\usepackage{enumitem}

\usepackage[textwidth=5cm]{todonotes}

\usepackage{todonotes}

\hypersetup{
  colorlinks=true,
  linkcolor=blue,
  urlcolor=red,
  pdftitle={}
}

\usepackage{amscd}

\usepackage{lineno}
\usepackage{tcolorbox}

\usepackage{hyperref} % hypelink
\hypersetup{
    colorlinks=true,
    linkcolor=blue,
    urlcolor=red,
    pdftitle={ },
    }

\newcommand\Vbullet{\raisebox{-2.1pt}{\kern-0.4pt\vbox{\baselineskip4pt\lineskiplimit0pt%
\hbox{$\bullet$}\hbox{$\bullet$}}}}

\newtheorem{theorem}{Theorem}[section]
\newtheorem*{theorem*}{Theorem}
\newtheorem{proposition}[theorem]{Proposition}

\newtheorem{lemma}[theorem]{Lemma}

\newtheorem{corollary}[theorem]{Corollary}
\newtheorem{definition}[theorem]{Definition}

\newtheorem{remark}[theorem]{Remark}

\newcommand{\C}{\mathbb{C}}

\newcommand{\RR}{\mathbb{R}}

\newcommand{\Z}{\mathbb{Z}}

\newcommand{\eps}{\varepsilon}

\newcommand{\vf}{\varphi}

\newcommand{\Q}{\mathbb{Q}}

\newcommand{\Li}{\mathrm{Li}}

\begin{document}

\title[General Stirling-Ramanujan Constants are exponential periods]
{General Stirling-Ramanujan Constants are exponential periods and applications}

\subjclass[2020]{40A05, 40G99, 11J81, 33B15, 11M35, 11B68, 11M06}
\keywords{Stirling-Ramanujan constants, exponential period, Ramanujan summation, Divergent series, Bendersky Gamma function, Dirichlet L-functions, determinant of Laplacians}

\author{Mounir Hajli
}
\address{Westlake University, Hangzhou, China
}
\email{hajlimounir@gmail.com
}

\author{Ricardo P\'erez-Marco
}
\address{CNRS, IMJ-PRG, Universit\'e Paris Cit\'e, B\^at. Sophie Germain, Paris, France.
}
\email{ricardo.perez.marco@gmail.com
}

\maketitle

\bigskip

\begin{abstract} For $n\geq 0$, Stirling-Ramanujan constants $S_n$ are the Ramanujan summation of the divergent series $\sum_{k\geq 1} k^n\log k$. These constants are exponential periods over the exponential base field $\mathbb{E}=\mathbb{Q}(t,e^{-t})$. We generalize this result to a broad class of General Stirling-Ramanujan constants. Given polynomials $P$ and $Q$, with $\Re Q(s)>0$ for $\Re s >0$, the constant $S(P,Q)$ is the Ramanujan summation of the series $\sum_{k\geq 1} P(k)\log Q(k)$. They are exponential periods over $\mathbb{K}_P((\omega)) (t,e^{-t}, (e^{-\omega t}))$ where $\mathbb{K}_P$ is the field of definition of $P$ and $(\omega)$ are the zeros of $Q$. These  constants are related to the derivatives $\zeta'_H(-n,w)$ of Hurwitz zeta function at negative integers, which we prove are exponential periods over $\mathbb{E} \left ( w,e^{-wt}\right )$. They can be expressed using Bendersky Gamma functions $\hat \Gamma_n$ and the values $\log \hat \Gamma_n(w)$ for $\Re w >0$ are exponential periods over  $\mathbb{E} \left ( w,e^{-wt}\right )$. The logarithms of the determinants of the Laplacian on spheres and lens spaces are exponential periods over  $\mathbb{E}$. The values  $\zeta(2n+1)/\pi^{2n}$ are exponential periods over $\mathbb{E}$. For a periodic function $\chi:\mathbb{Z}\to \mathbb{C}$, we extend Ramanujan summation to the twisted series $\sum_{k\geq 1} \chi(k) P(k)\log Q(k)$ and define General Twisted Stirling-Ramanujan constants $S_{\chi}(P,Q)$. We derive integral formulas proving that they are exponential periods over $\mathbb{E}(\chi )=\mathbb{Q}(\chi)(t,e^{-t})$.  For $n\geq 0$, $L_\chi'(-n)$ and $L_\chi(n+1)/\pi^n$ are given in terms of these constants and are exponential periods over $\mathbb{E}(\chi )$.

\end{abstract}

\newpage

\tableofcontents

\newpage

\section{Introduction}

For $n\geq 0$, the Stirling-Ramanujan constant $S_n$ is defined by the asymptotic expansion of the divergent series, when $s\to +\infty$,
\begin{equation*}
 \sum_{k=1}^s k^n \log k = A_n(s) \log s + B_n(s) +S_n + R_n(1/s)
\end{equation*}
where $S_n\in \RR$, $A_n\in \Q[s]$, $B_n\in s\Q[s]$, and $R_n(1/s)\in  \frac1s \Q[[\frac1s]]$.
This divergent series was studied by Ramanujan who defined
the constant $S_n$ (\cite[Chapter 9, p.\ 273]{Ber}).

\medskip

For $n=0$, $S_0=\frac{\log(2\pi)}{2}$ is the classical Stirling constant, and for $n=1$, $S_1=\log A$ where $A$ is the Glaisher-Kinkelin constant. Explicit integral formulas for Stirling-Ramanujan constants were found in \cite{MPM4}, proving that they are exponential periods.

\medskip
 Ramanujan summability procedure consists in identifying the constant term in the asymptotic expansion. The ``constant term'' is well defined in this context for an asymptotic expansion in the bases of functions $(s^m\log (s))_{m\geq 0} \cup (s^{m})_{m\in \Z}$. Such an asymptotic expansion, when it exists, is unique. For a series $\sum_{k=1}^s f(k)$ with such an expansion (as is the case for most series but some twisted ones considered in this article) we can define the Ramanujan sum as the constant term. We denote this constant by
$$
S_{\mathcal R} \left (f(k)\right )  \ .
$$
Ramanujan applied this procedure to a broad variety of series. Hardy attempted  to formalize unsuccessfully a general theory of  Ramanujan summability in his book on divergent series \cite{Ha2}. More recently Candelpergher proposes a rigorous theory in \cite{Can}, and denotes by
\[
\sum^{\mathcal R}_{k\geq 1} f(k)
\]
the constants provided by his theory.
An alternative approach using Poisson-Newton formula was proposed  in \cite{MPM1} section 5.5.

\subsection{{General Stirling-Ramanujan constants.}}

For polynomials $P,Q \in \C[X]$, with $\Re Q(k) >0$ for $k\geq 1$, the
General Stirling-Ramanujan constant
\[
S_{\mathcal R} (P(k)\log Q(k))
\]
is well defined as the divergent series has the above asymptotic expansion.
The study of these general constants can be reduced to the following particular case: Let
$n\geq 0$ and $w \in \C$ such that $\Re w > 0$, we define the constant $\mathbf{S}_n(w)$ associated with the divergent series
\[
 \sum_{k=0}^{s-1} (k+w)^n \log(k+w)
\]
Since we have
\[
 \sum_{k=0}^{s-1} (k+w)^n \log(k+w)= \sum_{k=1}^s (k+w-1)^n \log(k+w-1)
\]

then 
we have
\[
\mathbf{S}_n(w)=S_{\mathcal R} ((k+w-1)^n \log (k+w-1)) 
\]

In particular, for the case $w=1$, the value $\mathbf{S}_n(1) = S_n$ recovers the classical Stirling--Ramanujan constants studied in \cite{MPM4}.

\medskip

Our first main result provides an explicit integral representation for these constants.

\begin{theorem}\label{thm:integral-rep-Sv}
 For $n\geq 0$ and $w\in \C$ with $\Re w >0$, the constant $\mathbf{S}_n(w)$ is given by the integral:
\begin{equation*}
\mathbf{S}_n(w) = (-1)^{n+1} n! \left ( \int_0^{+\infty} \frac{1}{t^n} \left( \frac{1}{1-e^{-t}} - \sum_{k=-1}^{n} b_k t^{k} \right) \, \frac{e^{-wt} dt}{t}  +\mathbf{\hat r}_n(w)\right ),
\end{equation*}
where the rational coefficients $(b_k)_{k\geq -1}$ are given by the generating function
\[
 \frac{1}{1-e^{-t}} = \sum_{k=-1}^{+\infty} b_k t^{k}=\frac1t +\frac12+\frac{1}{12}\, t-\frac{1}{720} \, t^3+\ldots
\]
and the coefficient $\mathbf{\hat r}_{n}(w)$ is given by
\[
\mathbf{\hat r}_{n}(w) = \sum_{\ell=1}^{n+1} b_{n-\ell} \frac{(-1)^{\ell}}{\ell!} w^{\ell} H_{\ell} + \left ( \sum_{\ell=0}^{n+1} b_{n-\ell} \frac{(-1)^{\ell+1}}{\ell!} w^{\ell} \right )\log w,
\]
where the $H_n=1+\frac12+\ldots+\frac1n$ are the harmonic numbers (with $H_0=0$).
\end{theorem}

The additive term $\mathbf{\hat r}_{n}(w)$ can be incorporated inside the integral since
$$
\mathbf{\hat r}_{n}(w) =-\int_0^{+\infty} t\, \mathbf{r}_{n}(w) \, \frac{e^{-wt} dt}{t}=-\frac{\mathbf{r}_{n}(w)}{w},
$$
and this makes the link with the notation $ r_{n} =\mathbf{ r}_{n}(1)$ used in \cite{MPM4}.

If we introduce the modified Bernoulli polynomials defined by the generating series
 $$
 \frac{e^{-wt}}{1-e^{-t}} = \frac1t + \sum_{k=0}^{+\infty} b_k(w) t^k,
 $$

then we have
$$
b_n(w)=\sum_{\ell=0}^{n+1} b_{n-\ell} \frac{(-1)^{\ell}}{\ell !} w^{\ell},
$$
and we can write
$$
\mathbf{\hat r}_{n}(w) =b_n(w)\odot\left ( \frac{\Li_1(w)}{1-w}\right) - b_n(w) \log w,
$$

where $ \Li_1(w) =-\log(1-w)$ and $\odot$ denotes the Hadamard product. We recall that the generating series of harmonic numbers is given by
$$
\sum_{n=0}^{+\infty} H_n w^n = \frac{\Li_1(w)}{1-w} \ .
$$

Using the Higher Frullani integrals, one readily verifies that
the functions $w^n\log w$ are
exponential periods over the base field
\[
\Q(t,e^{-t},w,e^{-wt}).
\]
Hence $\hat{\mathbf r}_n(w)$ are also exponential periods over this same
base field and we obtain the following result.
\begin{corollary}\label{cor:base_field}
For $\Re w>0$, the General Stirling-Ramanujan constant $\mathbf{S}_n(w)$ is an exponential period over the field  $\Q(t,e^{-t},w,e^{-wt})$.
\end{corollary}

Observe that when $w=p/q\in \Q$, $q\geq 1$,  $\gcd(p, q)=1$, we have $\Q(e^{-t},e^{-wt}) =\Q(e^{-t/q})$.

\medskip

From this result we can derive the general formula for $S_{\mathcal R}(P(k)\log Q(k))$ using the
factorization of $Q$.  
If $b$ denotes the leading coefficient of $Q$, with $\Re(b)>0$, and the roots of $Q$ are of the form $-\omega$, with multiplicities $n_\omega$, we may assume, without loss of generality, that $\Re(\omega)>0$,
 and if $(a_n)_{0\leq n\leq d}$ are the coefficients of $P$,
\begin{align*}
P(X)&=\sum_{n=0}^d a_n X^n \ \ , \ \  Q(X)= b\prod_\omega (X+\omega)^{n_\omega},
\end{align*}
then we have
\begin{align*}
&\sum_{k=0}^{s-1} P(k) \log Q(k) = (\log b) \sum_{k=0}^{s-1} P(k)+ \sum_{n=0}^d \sum_\omega a_n n_\omega\sum_{k=0}^{s-1} k^n \log (k+\omega) \\
&=(\log b) \sum_{n=0}^d a_n \sum_{k=0}^{s-1} k^n+\sum_{n=0}^d \sum_\omega a_n n_\omega \sum_{j=0}^n \binom{n}{j} (-\omega)^{n-j}\sum_{k=0}^{s-1} (k+\omega)^j \log (k+\omega)\;.
\end{align*}
By Faulhaber's formula the first sum gives no contribution, therefore we have,
$$
S_{\mathcal R}(P(k-1) \log Q(k-1))= \sum_{n=0}^d \sum_\omega  \sum_{j=0}^n a_n n_\omega \binom{n}{j} (-\omega)^{n-j} \mathbf{S}_j(\omega),
$$
and by linear combination using Theorem \ref{thm:integral-rep-Sv} we have the integral formula for $S_{\mathcal R}(P(k-1) \log Q(k-1))$ as an exponential period. The inspection of the coefficients gives the
following Theorem.

\begin{theorem}\label{thm2}
The General Stirling-Ramanujan constant $S_{\mathcal R}(P(k-1) \log Q(k-1))$ is an exponential period over the base function field $\mathbb{K}_P((\omega)) (t,e^{-t}, (e^{-\omega t}))$
where $\mathbb{K}_P$ is the field of definition of $P$ and  
 $\mathbb{K}_P((\omega))$ is the splitting field  of $Q$ over $\mathbb{K}_P$.
\end{theorem}

\subsection{Relation to the Hurwitz zeta function.}

The constants $\mathbf{S}_n(w) $ are directly related to derivatives at
negative integer points of Hurwitz zeta function, defined for $\Re s >1$
and $\Re w >0$ by
$$
\zeta_H(s,w) =\sum_{k=0}^{+\infty} \frac{1}{(k+w)^s},
$$
and extends meromorphically to $s\in \C$, holomorphically to $s\in \C-\{1\}$. We
refer to \cite{WW} for more information about Hurwitz zeta function.

\begin{theorem}\label{thm:zetasn}
 Let $n\geq 0$ be an integer and $w\in \C$ with $\Re w >0$. We have
 \[
  \mathbf{S}_n(w) = -\zeta_H'(-n,w) +(-1)^{n+1} n! b_n(w) H_n \ ,
 \]
 where $\zeta_H'(s,w)= \frac{\partial}{\partial s} \zeta_H(s,w)$ and $b_{n}(w)$
 are the modified Bernoulli polynomials.

\end{theorem}
In particular, for $n=0$ we have
 \[
\mathbf{S}_0(w)= - \zeta_H'(0,w) \ .
 \]
 Also for $n\geq 1$ and $w=1$, we have $b_n(1)=b_n(0) = (-1)^{n+1} \frac{B_{n+1}}{(n+1)!}$ where $(B_{n})_n$ are the regular Bernoulli numbers. Therefore we have,
 \[
\mathbf{S}_n(1) = - \zeta'(-n)+(-1)^{n+1} n! b_n(1)H_n  =- \zeta'(-n)+ \frac{B_{n+1}}{n+1}H_n,
\]
where $\zeta$ is the Riemann zeta function.
 \begin{corollary} \label{cor:Adamchik_formula}
 We have
 $$
 \mathbf{S}_n(1) = - \zeta'(-n)+\frac{B_{n+1}}{n+1}H_n \; .
 $$
 \end{corollary}

This formula can be found in Adamchik's article \cite{Ad}.
The relation of the Stirling-Ramanujan constants (this is the case $w=1$)
to special values of Riemann zeta function is already present in Ramanujan's formulas (see \cite{Be} p.275). Our general formula for $w\in \C$, $\Re w >0$, generalizes all these formulas.

As a direct  Corollary we get an exponential period integral representation for
$\zeta_H'(-n,w)$ that seems new in the literature.
\begin{corollary}
We have
$$
\zeta_H'(-n,w) = (-1)^{n} n! \left ( \int_0^{+\infty} \frac{1}{t^n} \left( \frac{1}{1-e^{-t}} - \sum_{k=-1}^{n} b_k t^{k} \right) \, \frac{e^{-wt} dt}{t}  +\mathbf{\hat r}_n(w)-b_n(w) H_n\right)\;.
$$
The special value $\zeta'(-n,w)$ is an exponential period over the field
$\Q(t,e^{-t},w, e^{-wt})$.
\end{corollary}

As an immediate consequence, we obtain the following result for rational values of $w$.

\begin{corollary}
Let $n\geq 0$.

\begin{enumerate}
\item If $w\geq 1$ is an integer, then $
\zeta_H'(-n,w)$
is an exponential period over the base field $
\Q(t,e^{-t})\;.$

\item If $w=p/q\in\Q$, where $\gcd(p,q)=1$ and $q\geq 1$, then $
\zeta_H'\!\left(-n,\frac{p}{q}\right)$
is an exponential period over the base field
$
\Q(t,e^{-t/q})\;.$
\end{enumerate}
\end{corollary}

\subsection{Relation to Ramanujan-Candelpergher summation.}

We have a precise relation between our Ramanujan constants and
Ramanujan-Candelpergher summation (see Corollary \ref{cor:165_1}).

\begin{theorem} \label{thm:FormulaRC}
We have
\begin{align*}
\mathbf{S}_n(w) = \sum_{k \ge 1}^{\mathcal{R}} &(k+w)^n\log (k+w) +(-1)^{n+1} n! b_n(w) H_n +\\
&+\frac{(w+1)^{n+1}}{(n+1)^2}  - \frac{(w+1)^{n+1}}{n+1}\log(w+1)
+w^{n}\log w
\end{align*}
\end{theorem}

Observe that we have
$$
\sum_{k=0}^{s-1} (k+w+1)^n \log (k+w+1) -\sum_{k=0}^{s-1} (k+w)^n \log (k+w) =
(s+w)^n\log(s+w) -w^n\log w \ .
$$
In the asymptotics when $s\to +\infty$, the constant term
of $(s+w)^n\log(s+w)$ is the same as the one
$$
(s+w)^n\log \left (1+\frac{w}{s} \right ) = \left (\sum_{l=0}^n \binom{n}{l} s^l w^{n-l}
\right ). \left (-\sum_{k\geq 1} \frac{(-1)^k}{k} \left (\frac{w}{s}\right )^k \right )
$$
which is obtained for $k=l$,
$$
w^n\sum_{k=1}^n \binom{n}{k} \frac{(-1)^k}{k} = H_n w^n
$$
using the identity
$$
H_n=-\sum_{k=1}^n \binom{n}{k} \frac{(-1)^k}{k}  \ .
$$
Hence we have for $n\geq 0$,
$$
\mathbf{S}_n(w+1) -\mathbf{S}_n(w) =-w^n\log w+ w^n H_n \; .
$$
Taking the limit $w\to 0+$ we have $\lim_{w\to 0+} \mathbf{S}_n(w) = \mathbf{S}_n(1)$ for $n\geq 1$. Under natural assumptions, a Ramanujan-Candelpergher sum
depending on a parameter behaves analytically with respect to the parameter (see \cite{Can} Chapter 3, Theorem 9, p.62). In this case we apply the result at $w=1$. We obtain as Corollary:

\begin{corollary} \label{cor:w=1}
We have for $n\geq 1$,
$$
\mathbf{S}_n(1) =\lim_{w\to 0+} \mathbf{S}_n(w) = \sum_{k \ge 1}^{\mathcal{R}} k^n\log (k) +\frac{1}{(n+1)^2}+ \frac{B_{n+1}}{n+1}H_n \ ,
$$
and
$$
\sum_{k \ge 1}^{\mathcal{R}} k^n\log (k)  =-\zeta'(-n) - \frac{1}{(n+1)^2}
$$
\end{corollary}

This formula can be found in Coppo's article \cite{Cop} (see Theorem 1,
where in his notation $\mathbf{S}_n(1) =\log A_n$). Indeed, Coppo's formula
is a particular case of a previous more general formula by Candelpergher (top of page 77 of \cite{Can}). Candelpergher derived it directly (see also his formula at the bottom of page 91).

Coppo's proof relies on a formula that he attributes to Adamchik \cite{Ad}
and uses the Bendersky Gamma function and its properties \cite{Be}.
The formula attributed to Adamchik is
$$
\mathbf{S}_n(1) =-\zeta'(-n) + \frac{H_n B_{n+1}}{n+1}
$$
appears in earlier literature, for instance in  Choudhury \cite{Ch}\footnote{We are indebted to Coppo for providing this reference that he discovered after publishing his article.},
and was probably well known to Ramanujan in  view of the formulas appearing in his notebooks (\cite{Be} p.276).

\subsection{Relation to Bendersky Gamma function.}

Bendersky defined and studied in \cite{Be} a hierarchy of Gamma functions $(\hat \Gamma_n)_{n\geq 0}$
that satisfy the functional equation, for $n\geq 0$ and $\Re w>0$,
$$
\hat \Gamma_n(w+1) = w^{w^n} \hat \Gamma_n(w) \ ,
$$
with $\hat \Gamma_n(1)=1$.
We have that $\hat \Gamma_0$ is Euler Gamma function. For $n\geq 0$, $\hat \Gamma_n$ is
holomorphic and without zeros for $\Re w>0$ and we can consider $\log \hat \Gamma_n(w)$
for $\Re w >0$ (using the principal branch of the logarithm function at $w=1$).
Much earlier than Bendersky, in 1860 Kinkelin \cite{Ki} defined and
studied $\hat \Gamma_1$, that he named $G$,
and hinted at the interest of studying the hierarchy $(\hat \Gamma_n)_{n\geq 0}$.
These Bendersky Gamma functions are not to be
confused with the other hierarchy of meromorphic Gamma functions due to Barnes.
Unfortunately, both Kinkelin and Barnes (\cite{Ba}) gave the same
name $G$ to the first function in their
hierarchy that is not Euler Gamma function. This created some confusion in the literature.

Also, we have for $n\geq 0$,
\begin{equation}\label{eq:1}
\log \hat \Gamma_n(w)=\zeta'_H(-n, w) -\zeta'(-n) \ .
\end{equation}
It is unfortunate that this Lerch type formula was missing
from Bendersky's memoir. Its omission created
more confusion.  Milnor \cite{Mi}
defined a generalization of the Gamma function using the
partial derivative for Hurwitz zeta function, which is
the above formula (\ref{eq:1})  without
the constant term $-\zeta'(-n)$. He was not aware of the prior
construction of these functions by Bendersky 50 years before.
As far as the authors know, the identification of Milnor's Gamma function
with Bendersky's Gamma function was only pointed out recently by the second author
in \cite{PM} (see last section). Several authors continued to use
Milnor's Gamma function without
being aware of Bendersky's Gamma hierarchy. An example of this is the
article by  Kurokawa and Ochiai \cite{KuOch} where the above
formula was established for
a normalization of Milnor Gamma function (see Theorem 2).

Now, using the previous relation from Theorem \ref{thm:zetasn} of
the generalized Stirling-Ramanujan constants
$S_n(w)$ with Hurwitz zeta function, we have:

\begin{theorem}
 Let $n\geq 0$ be an integer and $w\in \C$ with $\Re w >0$. We have
$$
\mathbf{S}_n(w) =-\log \hat \Gamma_n(w) -\zeta'(-n) + (-1)^{n+1} n! \, b_n(w)\, H_n \ .
$$
\end{theorem}

Using this result we can construct Bendersky's Gamma function in a very direct
way by using the construction of Generalized Stirling-Ramanujan
constants $\mathbf{S}_n(w)$. Another Corollary using formula (\ref{eq:1}) is

\begin{corollary}
Let $n\geq 0$.
\begin{enumerate}
\item
For $\Re w >0$, $\log \hat \Gamma_n(w)$ is an exponential period over the base field $
\Q(t,e^{-t}, w, e^{-wt})$.
\item If $w\geq 1$ is an integer, then $\log \hat \Gamma_n(w)$
is an exponential period over the base field $
\Q(t,e^{-t})$.

\item If $w=p/q\in\Q$, where $\gcd(p,q)=1$ and $q\geq 1$, then $
\log \hat \Gamma_n(p/q)$
is an exponential period over the base field
$\Q(t,e^{-t/q})$.
\end{enumerate}
\end{corollary}

\subsection{Other applications.}

\subsubsection{Determinants of Laplacians.}
As Corollary of the previous results, we obtain some transalgebraic information about
the classical determinants of Laplacians.

Let $(M, g)$ be a compact, connected Riemannian manifold of dimension $d$. Let $\Delta_M$ denote the Laplace-Beltrami operator acting on smooth functions on $M$ associated with the metric $g$. The operator $\Delta_M$ possesses a non-negative, unbounded spectrum
\[ \text{Spec}(\Delta_M) = \{0 = \lambda_0 < \lambda_1 \le \lambda_2 \le \cdots \}, \]
to which it is associated the spectral zeta function defined for $\text{Re}(s) > \frac{d}{2}$ by the series
\[ \zeta(s, \Delta_M) = \sum_{k=1}^\infty \lambda_k^{-s}\;. \]
 It is a classical result that $\zeta(s, \Delta_M)$ extends to a meromorphic function on $\mathbb{C}$ that is analytic at $s = 0$. The \textit{zeta-regularized determinant} of $\Delta_M$ is defined as
\[ \det\nolimits \Delta_M = \exp\left( -\zeta'(0, \Delta_M) \right). \]

 Zeta-regularized determinants carry essential geometric information and have been studied extensively.  However, explicit computations are known only in relatively few cases where
 the spectrum is known. For instance, Kumagai \cite{Kum} obtained explicit formulas for the regularized determinant $\det\nolimits \Delta_{\mathbb{S}^d}$ of the $d$-dimensional sphere endowed with its standard metric coming from its embedding in the Euclidean space $\mathbb S^d\subset \RR^{d+1}$. In Kumagai's formulas \cite{Kum}, there are explicit integers 
 $(\alpha_{d,j})_{0\leq j\leq d}$ and $(\tau_{d,j})_{0\leq j\leq d}$, and a rational
 number $\gamma_d$, such that
 $$
 \log \det\nolimits \Delta_{\mathbb{S}^d} = \gamma_d +\sum_{k=1}^d \alpha_{d,k} \log k +  \sum_{k=0}^{d} \tau_{d,k} \zeta'(-k),
 $$
The Frullani integral
$$
\log (s+1) =\int_0^{+\infty} (1-e^{-st}) \, \frac{e^{-t} dt}{t}
$$
shows that for any integer $k\geq 2$, $\log k$ is an exponential period over the base
field $\Q(t,e^{-t})$. Using our previous results we also have that $\zeta'(-k)$
is also an exponential period over the base field $\Q(t,e^{-t})$. This provides new
insight into the transalgebraic nature of these determinants of Laplacians:

\begin{theorem}\label{thm:log-determinants_are_periods}
For every integer $d\geq 1$, the logarithm of the zeta-regularized determinant of the Laplacian on the sphere $\log \det\nolimits \Delta_{\mathbb S^d}$
is an exponential period over the field $\mathbb{Q}(t,e^{-t})$.
\end{theorem}

It is interesting to note that the base field $\mathbb{Q}(t,e^{-t})$ does
not depend on the dimension $d\geq 1$.
 Other explicit computations for spheres and   projective spaces  have been treated in  \cite{Ha} and  \cite{KuKo} using different methods, with the formulas
 having the same form, so the result generalizes.

We can write down the explicit period integrals.  For dimension $d=1,2, 3$
we have the following formulas:

\begin{proposition}\label{prop:examples}
In dimensions $d=1,2,3$ we have the explicit representation of $\log \det\nolimits \Delta_{\mathbb S^d}$
as an exponential period over the field $\mathbb{Q}(t,e^{-t})$,

\begin{align*}
\log \det \Delta_{\mathbb S^1} &= -4
\int_0^{+\infty}
\left(
\frac{1}{1-e^{-t}}
-\frac1t
-\frac12
-t
\right)
\frac{e^{-t}}{t}\,dt \\
\log \det \Delta_{\mathbb S^2}
&=
4
\int_0^{+\infty}
\frac1t
\left(
\frac{1}{1-e^{-t}}
-\frac1t
-\frac12
-\frac{t}{12}
+\frac{7}{24}t^2
\right)
\frac{e^{-t}}{t}\,dt\\
\log \det \Delta_{\mathbb S^3}
&=-2\int_0^\infty \left(\left(1+\frac{2}{t^2}\right)\frac{1}{1-e^{-t}}
-\frac{2}{t^3}
-\frac{7}{6t}
-\frac{1}{t^2}
-\frac{37}{36}t
-\frac{1}{2}e^{-t} \right)\frac{e^{-t}dt}{t}\;.
\end{align*}

\end{proposition}

The phenomenon observed in Theorem~\ref{thm:log-determinants_are_periods} is not isolated. It also occurs for more general manifolds with rich geometric structures, such as symmetric spaces and their quotients by finite groups. To illustrate this, we consider the first natural generalization of the result for spheres, namely lens spaces. These are quotients of odd-dimensional spheres by cyclic groups.

\begin{theorem}\label{thm:lens}
Let $L(q;p_0,\ldots,p_n)$ be a lens space. Then the logarithm of the zeta-regularized determinant of the Laplacian on $L(q;p_0,\ldots,p_n)$ is an exponential period over the field
\[
\mathbb{Q}(t,e^{-t}).
\]
\end{theorem}

Note that the statement does not depend on the particular cyclic group.

\subsubsection{On the special values $\zeta(2n+1)/\pi^{2n}$.}

The values of the Riemann zeta function at odd values normalized by $\pi^{2n}$,
namely $\zeta(2n+1)/\pi^{2n}$ are directly related to Stirling-Ramanujan
sums $\mathbf S_{2n}(1)$. It results that
these values are exponential periods.

\begin{theorem}\label{thm:zeta(2n+1)}
 For $n\geq 1$, we have that $\zeta(2n+1)/\pi^{2n}$ are exponential
 periods over the base field
 $\mathbb Q(t, e^{-t})$. More precisely, we have
$$
\frac{\zeta(2n+1)}{\pi^{2n}} = (-1)^n 2^{2n+1} \left ( \int_0^{+\infty} \frac{1}{t^{2n}} \left( \frac{1}{1-e^{-t}} - \sum_{k=-1}^{2n} b_k t^{k} \right) \, \frac{e^{-t} dt}{t} +c_n \right )
 $$
where $c_n \in \mathbb Q$ is given by
$$
c_n= \sum_{\ell=1}^{2n+1} b_{2n-\ell} \frac{(-1)^{\ell}}{\ell!} H_{\ell}  \ .
$$
\end{theorem}

\subsection{Twisted Stirling--Ramanujan constants}

 Let $\chi:\mathbb Z\to \mathbb C$ be a periodic function of period $m$. Given two polynomials
\[
P,Q\in\C[s],
\]
with $\deg Q\geq1$ and leading coefficient of $Q$ having positive real part, we consider divergent series of the form
\[
\sum_{k=1}^{ms}\chi(k)P(k)\log Q(k).
\]

The first reduction consists in assuming that $\chi(0)=0$. Indeed,
\[
\sum_{k=1}^{ms}\chi(k)P(k)\log Q(k)
=
\sum_{k=1}^{ms}(\chi(k)-\chi(0))P(k)\log Q(k)
+
\chi(0)\sum_{k=1}^{ms}P(k)\log Q(k),
\]
and the asymptotic expansion of the second sum has already been studied. Hence, in the sequel, we assume
\[
\chi(0)=0 \ .
\]

 We have
\[
 \sum_{k=1}^{ms} \chi(k)\, P(k)\log Q(k)=\sum_{a=0}^{m-1} \chi(a)\sum_{\ell=0}^{s-1} P(a+\ell m)\log Q(a+\ell m)\;.
\]
It is natural to introduce the following definition.

\begin{definition}
Let $\chi$ be a periodic function of period $m\geq1$ with $\chi(0)=0$. The associated twisted Stirling--Ramanujan constant is defined by
\[
 S^{(P,Q)}_{\chi,m}
 :=
 \sum_{a=0}^{m-1}\chi(a)\,S^{(P,Q)}_{m,a}.
\]
where $S^{(P,Q)}_{m,a}=S_{\mathcal{R}}(P(a+(k-1)m)\log Q(a+(k-1)m))$\;.
\end{definition}

The terms $A(s)\log s+B(s)$ in the asymptotic expansion are periodic with period $m$ and play no role in the Ramanujan summation procedure.

\medskip

The purpose of this section is to prove that these twisted Stirling--Ramanujan constants are exponential periods. As in the untwisted case, the general situation reduces to a distinguished family of constants. More precisely, we consider the twisted constants
\[
\hat{S}^{(k+w)^n,(k+w)}(\chi)
\]
associated with the asymptotic expansion of
\[
\sum_{k=1}^{ms}\chi(k)(k+w)^n\log(k+w).
\]

Our main result, Theorem~\ref{thm:105_2}, shows that these constants are exponential periods over the field $
\Q(\chi)(t,e^{-t},w,e^{-wt})\;.$

\medskip

We also consider the special constants $\hat{S}_{n,0}(\chi)$
associated with the series
\[
\sum_{k=1}^{ms}\chi(k)\,k^n\log k.
\]
By definition,
\[
\hat{S}_{n,0}(\chi)=\hat{S}^{k^n,k}(\chi).
\]
Let $\chi:\Z\to\C$ be a periodic function of period $m$ satisfying $\chi(0)=0$. We obtain the following integral representation.

\begin{theorem}\label{thm:twisted_formula}
The constant $\hat{S}_{n,0}(\chi)$ admits the representation
\[
\begin{aligned}
\hat{S}_{n,0}(\chi)
&=
(-1)^{n+1}n!
\int_0^{+\infty}
\frac{1}{t^n}
\left(
\frac{F_{\chi}(-t,-1)}{t}
-
\sum_{k=-1}^{n} b_{\chi,k}t^k
\right)
\frac{e^{-t}\,dt}{t}
\\
&\quad
-
(-1)^{n+1}n!
\sum_{j=0}^{n+1}
\frac{(-1)^{j+1}}{j!}
H_j\, b_{\chi,n-j}
\\
&\quad
+
(-1)^{n+1}n!\,
b_{\chi,n}(1)\log m.
\end{aligned}
\]
where
\[
F_{\chi}(t,x)
=
\sum_{a=1}^{m}
\frac{\chi(a)\,t\,e^{at}}{e^{mt}-1}\,e^{xt},
\]
and the coefficients $b_{\chi,k}$, for $k=-1,0,1,\ldots$, are defined by the generating series
\[
\frac{F_{\chi}(-t,-1)}{t}
=
\sum_{k=-1}^{\infty} b_{\chi,k}\,t^k.
\]
\end{theorem}

For the twisted Bernoulli numbers $b_{\chi,k}$, twisted Bernoulli polynomials and properties of $F_\chi$, we refer to the Appendix \ref{appendix:Bernoulli} and to \cite{Co}. When $\chi$ is a character the definition is more classical (see \cite{Iwa}).

When $\chi$ is the constant function, the integral formula reduces to the integral representation of the classical Stirling--Ramanujan constants.

\begin{corollary}
	The constant $\hat{S}_{n,0}(\chi)$ is an exponential period over the base field $\Q(\chi)(t,e^{-t})$.
\end{corollary}

In general, for $\ell=0,1,\ldots,m-1$, we define $\hat{S}_{n,\ell}(\chi)$ by the asymptotic expansion
\[
 \sum_{k=1}^{ms+\ell}\chi(k)\,k^n\log k
 =
 \hat{A}_{\chi,n,\ell}(s)\log s
 +
 \hat{B}_{\chi,n,\ell}(s)
 +
 \hat{S}_{n,\ell}(\chi)
 +
 \hat{R}_{\chi,n,\ell}\!\left(\frac1s\right).
\]

Then
\[
\hat{S}_{n,\ell}(\chi)
=
\hat{S}_{n,0}(\chi)
+
\hat{s}_{n,\ell}(\chi),
\qquad
\ell=0,1,\dots,m-1,
\]
where
\[
\hat{s}_{n,\ell}(\chi)
=
(H_n+\log m)\sum_{k=1}^{\ell}\chi(k)k^n.
\]

\begin{corollary}
The twisted Stirling--Ramanujan constants
\[
\hat{S}_{n,0}(\chi),\ldots,\hat{S}_{n,m-1}(\chi)
\]
are exponential periods over the base field $\Q(\chi)(t,e^{-t})$.
\end{corollary}

Theorem \ref{thm:twisted_formula} is obtained from the parameter 
Theorem \ref{thm:105_2} as in the non-twisted case. 

\subsection{Application to special values of Dirichlet $L$-functions}

When $\chi$ is a Dirichlet character, we consider the associated Dirichlet $L$-function $L_\chi$.
The twisted
Stirling--Ramanujan constants $\hat{S}_{n,0}(\chi)$ are directly related to derivatives
of Dirichlet $L_\chi$ at non-positive integers. 

\begin{theorem}\label{cor:205_1}
	For every $n\geq 0$, one has
	\[
	\begin{aligned}
		\hat{S}_{n,0}(\chi)
		&=
		-L_\chi'(-n)
		+
		(-1)^{n+1}n! \, m^n
		\left(
		\sum_{a=1}^{m} \chi(a)\, b_n\!\left(\frac{a}{m}\right)
		\right) H_n
		\\
		&\quad
		+
		\frac{B_{\chi,n+1}}{n+1}\,  \log m.
	\end{aligned}
	\]
\end{theorem}
Note that the result remains valid for every periodic function $\chi$ with $m\geq 1$.
From this, as a direct corollary, we get information about the transalgebraic nature of these 
 special values of $L$-functions : 
\begin{corollary}
For $n\geq 0$, $L_\chi'(-n)$ is an exponential period over the base field $\Q(\chi)(t,e^{-t})$.
\end{corollary}

\begin{theorem}\label{thm:dirichletperiod}
Let $\chi$ be a Dirichlet character modulo $m$, and let $\chi_1$ be the
primitive character of conductor $m_1$ inducing $\chi$. Write
\[
\chi(-1)=(-1)^a.
\]
Then, for $n\geq1$ with $n\equiv a\pmod{2}$,
\[
\frac{L_\chi(n+1)}{\pi^n}
\]
is an exponential period over the base field $
\mathbb Q(\chi,\zeta_{4m_1})(t,e^{-t}),$
where $\zeta_{4m_1}$ is a primitive $4m_1$-th root of unity\;.
\end{theorem}

\subsection{Structure of the proofs.}
Our fundamental result is Theorem \ref{thm:integral-rep-Sv}.
It can be proved in different ways. A first direct approach
consists in following the same ideas as in \cite{MPM4}.
This  gives an integral formula for the divergent sum by
using Higher Frullani integrals, and manipulating the integral
into three distinct parts, one for the divergent part of the asymptotic
expansion, another for the constant term and a last one for
the vanishing part of the asymptotics. A simplified example of this
type of manipulation is provided in the proof Proposition
\ref{prop:Formula_n=-1} below. The resulting computations in \cite{MPM4} are
laborious, and in the general case considered in this article, such an
approach would become even more cumbersome and not enlightening for the
reader. On the other hand, the advantage of this approach is that it is
very direct and fully self-contained.

Instead, we have decided to follow a different path, with the goal
to minimize the amount of computations and gain new insights into the
General Stirling-Ramanujan constants. The advantage of this approach 
is its simplicity which reduces the computational burden. The core strategy
consists  of exploiting a parameter dependence by introducing the natural
parameter $w$. The case of $w=1$ corresponds to the case treated in \cite{MPM4}.
Integration along the parameter allows to climb the ladder of the integer parameter $n$.
On the other hand, this approach is not self-contained because in the
integration procedure we need  to know the  integral formula for a particular
value $w=w_0$. In that particular point we use the result from \cite{MPM4}
for the value $w_0=1$. Also, we use Euler-Maclaurin formula to infer
the form of the asymptotics (this comes for free in the first approach).

Choosing this path forces us to structure the proofs in a particular way.
In particular, we are led to treat first the non-twisted case. Then, by a
standard procedure, we unwind a twisted series into a linear combination
of non-twisted ones. In this way we can reach the most general results,
without any painful computation.

We start with the proof of Theorem \ref{thm:integral-rep-Sv}.

\part{}

\section{Proof of Theorem \ref{thm:integral-rep-Sv}.}

\subsection{Preliminaries.}
Given an integer $n\geq 0$ and a complex parameter $w\in \C$, $\Re w >0$, we study
the divergent series
\[
 \sum_{k=0}^{s-1} (k+w)^n \log(k+w) \ .
\]
Our first task is to establish  an asymptotic expansion of the appropriate form, so that the Ramanujan constant
$\mathbf{S}_n(w)$ of this series is well defined as the constant term.

It is convenient to define constants $\mathbf{S}_n(w)$ also for negative $n\leq -1$, but this time the logarithmic term being absent from the series. More specifically, for $n\leq -1$, we define $\mathbf{S}_n(w)$ as the Ramanujan summation of the series
\[
 \sum_{k=0}^{s-1} (k+w)^n.
\]
Note that using the same notation $\mathbf{S}_n(w)$ for both the logarithmic and non-logarithmic series is an abuse of notation. This convention is justified by the recurrence relations in $n$. We note that for $n \leq -2$, the series converges, and the Ramanujan constant coincides with the infinite sum
$$
 \mathbf{S}_n(w) = \sum_{k=0}^{\infty} (k+w)^n = \zeta_H(-n, w) \ .
$$

\subsection{Asymptotics.}

\subsubsection{Asymptotics for $n\geq 0$.}Our  definition of the Stirling--Ramanujan constant relies on the asymptotic behavior of the partial sums.
A direct  application of Euler-Maclaurin formula gives the following asymptotics:
\begin{lemma}\label{lemma:132_5}
Let $n\geq 0$ and $w \in \mathbb{C}$ with $\Re w>0$. When $s \to \infty$, we have the asymptotic expansion,
\[
 \sum_{k=0}^{s-1} (k+w)^n \log(k+w) = \mathbf{A}_{n}(s,w) \log s + \mathbf{B}_n(s,w) + \mathbf{S}_n(w) + \mathbf{R}_n(w, 1/s),
\]
where $\mathbf{S}_n(w)\in \C$, $\mathbf{A}_{n}(s,w) \in \mathbb{Q}[s,w]$ and $\mathbf{B}_{n}(s,w) \in s(\mathbb{Q}[w,s] \oplus (\log w) \mathbb{Q}[w,s])$ are polynomials in $s$, and $\mathbf{R}_n(w, 1/s) \in \  \frac1s \mathbb{Q}[w][[1/s]]$.

When we consider $w$ as an independent variable,
the asymptotic expansion in this base of functions is unique.
\end{lemma}

\begin{proof}We treat $w$ as an independent variable in the complex half plane $\Re w>0$.
The asymptotic follows from Euler-Maclaurin summation formula applied to the function $f_n(z, w) = (z+w)^n \log(z+w)$ with respect to the variable $z$. We recall Euler-Maclaurin summation formula with remainder to order $m\geq 1$ (see for example \cite{Can}, p.178),
\begin{align*}
\sum_{k=\alpha}^\beta f(k)=& \int_\alpha^\beta f(t)\, dt +\frac12 (f(\alpha)+f(\beta))
+\sum_{k=1}^m \frac{B_{2k}}{(2k)!} ( f^{(2k-1)}(\beta) - f^{(2k-1)}(\alpha)) +R_m
\end{align*}
with
\[
R_m= \frac{1}{(2m+1)!} \int_\alpha^\beta B_{2m+1}(t-[t]) f^{(2m+1)} (t) \, dt
\]
where $(B_k)$ are the Bernoulli numbers  that are rational and $B_{2m+1}$ the Bernoulli polynomials  with rational coefficients. We observe that the primitive of $(z+w)^n \log(z+w)$ is in $\Q[z, w] \log (z+w)$. Furthermore, for any $m \geq 1$, the derivatives satisfy,
\[
\frac{\partial^{m}}{\partial z^{m}} \left( (z+w)^n \log(z+w) \right) \in \mathbb{Q}[z,w] \log(z+w) \oplus \mathbb{Q}[z,w](z+w)^{n-m}.
\]
For $m > n$, the logarithmic term vanishes in the higher-order derivatives, and the expression becomes a rational function in $(z+w)$. Specifically, for sufficiently large $m$, the derivatives are of the form $\mathbb{Q}[z,w] (z+w)^{-(m-n)}$, which admits an expansion in the power basis $\frac{1}{z}$ as:
\[
(z+w)^{-\ell} \in \frac{1}{z^\ell} \mathbb{Q}[w][[1/z]]\;.
\]
Applying the formula with $\alpha = 0$ and $\beta = s-1$ and collecting terms according to their growth in $s$ yields the desired form.

We have uniqueness of the asymptotic expansion because we have a total order relation of the functions in the bases by growth domination at infinity. In particular this implies that they are $\C$-linearly independent and so $\Q$-linearly independent.

More specifically, we prove the uniqueness by considering two asymptotics for the same function. Subtracting and looking at the leading term for $\Re s\to +\infty$ it must be zero, hence both asymptotics are equal.

\end{proof}

\begin{remark}
When $w=1$, these asymptotics appear explicitly in \cite{MPM4}. They were directly derived, without using Euler-Maclaurin formula, by using Higher Frullani integrals. \end{remark}

\begin{remark} The form of the asymptotics in the bases of functions splits into four complementary subspaces
$$
\bigoplus_{n\geq 0} \Q[w].(s^n\log s) \bigoplus \left (s(\mathbb{Q}[w,s] \oplus (\log w) \mathbb{Q}[w,s]) \right )\bigoplus C_w \bigoplus_{n\geq 1} \C[w]. (s^{-n}),
$$
where $C_w$ is the space of holomorphic functions depending only on $w$, for $\Re w >0$,  and independent of $s$.
An important observation is that each subspace is stable by the differential operator $w\frac{d}{dw}$.
\end{remark}

\subsubsection{Asymptotics  for $n\leq -1$.}

We complement Lemma \ref{lemma:132_5} with the cases for $n = -1$ and $n \leq -2$.

\begin{lemma}
For $n=-1$ we have the asymptotics, for $ s \to +\infty$,
$$
\sum_{k=0}^{s-1} (k+w)^{-1}  =\log s + \mathbf{S}_{-1}(w) +\mathbf{R}_{-1}(w,1/s),
$$
and for $n\leq -2$, the sum converges and we have the asymptotics for $\Re s \to +\infty$,
$$
\sum_{k=0}^{s-1} (k+w)^{n}  = \mathbf{S}_n(w) +\mathbf{R}_{n}(w,1/s),
$$
where in both cases $\mathbf{R}_n(w, 1/s) \in \  \frac1s \mathbb{Q}[w][[1/s]]$.
\end{lemma}
The proof follows from the Euler-Maclaurin formula applied to the function $f_n(z, w) = (z+w)^n$.

In conclusion, all the constants $\mathbf{S}_n(w)$ are well defined for all $n\in \Z$.

\subsection{Integral formula for $n\leq -1$.}
In this simple case, we can compute directly  the integral formula, which provides the initial step for the general case.

\begin{lemma}
For $n\leq -1$, the partial sums admit the following integral representation,
$$
\sum_{k=0}^{s-1} (k+w)^{n} =\frac{1}{ (-n-1)! } \int_0^{+\infty}  t^{-n} \frac{1-e^{-st}}{1-e^{-t}} \, \frac{e^{-wt} dt}{t}
$$
\end{lemma}

\begin{proof}
For the case $n=-1$, using
$$
(k+w)^{-1} = \int_0^{+\infty} e^{-(k+w)t} \, dt,
$$
we compute
\begin{align*}
\sum_{k=0}^{s-1} \frac{1}{k+w} = \sum_{k=0}^{s-1} \int_0^{+\infty} e^{-(k+w)t} \, dt =\int_0^{+\infty} t\, \frac{1-e^{-st}}{1-e^{-t}} \, \frac{e^{-wt} dt}{t}\;.
\end{align*}
For the general case $n\leq -2$, we take $-n-1$ derivatives with respect to $w$, in the previous formula we get,
\[
(-1)^{-n-1} (-n-1)! \sum_{k=0}^{s-1} (k+w)^n =  (-1)^{-n-1} \int_0^{+\infty} t^{-n}\, \frac{1-e^{-st}}{1-e^{-t}} \, \frac{e^{-wt} dt}{t}\;.
\]

\end{proof}

We note that for $n\leq -2$ we have $S_{n}(w)=\zeta_H(-n,w)$. In this simpler 
case, making $s\to +\infty$ in the previous integral formula, 
yields the integral formula for $\mathbf{S}_{n}(w)$.

\begin{proposition}
For $n\leq -2$ we have
\begin{align*}
\mathbf{S}_{n}(w) &= \frac{1}{ (-n-1)! } \int_0^{+\infty}  t^{-n} \, \frac{1}{1-e^{-t}} \, \frac{e^{-wt} dt}{t}\\
&= \frac{1}{ (-n-1)! } \left (\int_0^{+\infty} t^{-n} \left (\frac{1}{1-e^{-t}} -\frac1t\right )\, \frac{e^{-wt} dt}{t} + (-n-2)! w^{n+1}\right )\;.
\end{align*}
\end{proposition}

In the case $n=-1$ we can also get the integral formula with some little  manipulation that is the simplest case of the proof in \cite{MPM4}. In this case, this is the well-known asymptotic for the harmonic series that yields Euler-Mascheroni constant (although the derivation in the textbooks is usually different).

\begin{proposition}\label{prop:Formula_n=-1}
For $n=-1$, we have
$$
\mathbf{S}_{-1}(w) = \int_0^{+\infty}  t\, \left (\frac{1}{1-e^{-t}} -\frac1t \right ) \, \frac{e^{-wt} dt}{t}-\log w\;.
$$
\end{proposition}

Observe that for $w=1$ we have the classical integral formula for the Euler-Mascheroni constant $\gamma$ and  $S_{-1}(1)=\gamma $.

\begin{proof}
We use the integral formula
\begin{align*}
 &\sum_{k=0}^{s-1} \frac{1}{k+w} = \int_0^{+\infty}  t \,  \frac{1-e^{-st}}{1-e^{-t}} \, \frac{e^{-wt} dt}{t}\\
 &=- \int_0^{+\infty} \left (  (e^{-st}-1) +e^{-st} t \left ( \frac{1}{1-e^{-t}} -\frac1t \right ) - t \left (\frac{1}{1-e^{-t}}-\frac1t  \right ) \right ) \, \frac{e^{-wt} dt}{t}\\
 &= \log \left (\frac{s}{w} +1\right ) + \int_0^{+\infty}  t\, \left (\frac{1}{1-e^{-t}} -\frac1t \right ) \, \frac{e^{-wt} dt}{t} + \mathcal O(1/s)\\
 &= \log s + \left (  \int_0^{+\infty}  t\, \left (\frac{1}{1-e^{-t}} -\frac1t \right ) \, \frac{e^{-wt} dt}{t} -\log w \right ) +\mathcal O(1/s),
\end{align*}
where we have used the Frullani integral for the logarithm
\begin{align*}
\log \left (\frac{s}{w} +1\right ) &= \int_0^{+\infty} (1-e^{-st}) \, \frac{e^{-wt}}{t} dt \\
&=\log s -\log w+\mathcal{O}(1/s),
\end{align*}
and the fact that
$$
\int_0^{+\infty} \left ( \frac{1}{1-e^{-t}} -\frac1t \right )e^{-wt} \,   e^{-st} dt = \mathcal O(1/s),
$$
as we recognize the Laplace transform of a regular function at $t=0$.

\end{proof}
We conclude this section by observing  that for $n\leq -1$ the constants satisfy
 $$
 \frac{d}{dw}  \mathbf{S}_n(w) = n \, \mathbf{S}_{n-1}(w)
 $$
In the next section we generalize this differential formula for all $n\in \Z$ and derive from it the general integral formula for $\mathbf{S}_n(w)$.

\subsection{Recurrence.}
\subsubsection{Differential recurrence.}
\begin{proposition}
 For $w\in \mathbb C$ with $\Re(w)>0$, and $n\geq 1$, we have
 $$
 \frac{d}{dw}  \mathbf{S}_n(w) = n \, \mathbf{S}_{n-1}(w)
 $$
By iteration, it follows that for any $n \geq 1$,
 \[
\frac{d^n}{dw^n} \mathbf{S}_n(w) = n! \, \mathbf{S}_{0}(w).
\]
For $n=0$ we have
\[
 \frac{d}{dw}  \mathbf{S}_0(w) =  \mathbf{S}_{-1}(w)
\]
\end{proposition}

\begin{proof}

For $n\geq 1$ we apply the differential operator $w d/dw$,
\begin{align*}
 &w \frac{d}{dw} \left( \sum_{k=0}^{s-1} (k+w)^n \log(k+w) \right) = nw \sum_{k=0}^{s-1} (k+w)^{n-1} \log(k+w) + w\sum_{k=0}^{s-1} (k+w)^{n-1} \\
 &= n w\sum_{k=0}^{s-1} (k+w)^{n-1} \log(k+w) + w\frac{B_n(s-1+w) - B_n(w-1)}{n},
\end{align*}
where $B_n(x)$ is the $n$-th Bernoulli polynomial, and by Faulhaber's formula
$\frac{w}{n}\left(B_n(s-1+w) - B_n(w-1)\right) \in s\Q[s,w]$. The uniqueness and the invariance of the three non-trivial subspaces of the asymptotic bases by the operator $w d/dw$, prove that the derivative of the constant term $S_n(w)$ coincides with the constant term of the series for $n-1$ multiplied by $n$.

For $n=0$, we have
$$
 \frac{d}{dw} \left( \sum_{k=0}^{s-1} \log(k+w) \right) =  \sum_{k=0}^{s-1} \frac{1}{k+w}
$$
and the same argument proves that
\[
 \frac{d}{dw}  \mathbf{S}_0(w) =  \mathbf{S}_{-1}(w)
\]
\end{proof}

\subsubsection{Proof of Theorem  \ref{thm:integral-rep-Sv}}

We define for $n\geq  -1$, and $\Re w >0$ the holomorphic function
$\mathbf{\hat S}_n(w)$ by
\begin{equation*}
\mathbf{\hat S}_n(w) = (-1)^{n+1} n! \left ( \int_0^{+\infty} \frac{1}{t^n} \left( \frac{1}{1-e^{-t}} - \sum_{k=-1}^{n} b_k t^{k}  \right) \frac{e^{-wt}}{t} \, dt +\mathbf{\hat r}_n(w)\right )\quad \text{for } n\geq 0
\end{equation*}
and
$$
\hat{\mathbf{S}}_{-1}(w) = \int_0^{+\infty}  t\, \left (\frac{1}{1-e^{-t}} -\frac1t \right ) \, \frac{e^{-wt} dt}{t}-\log w\;.
$$

where
\begin{equation}\label{eq:rn}
\mathbf{\hat r}_{n}(w) = \sum_{\ell=1}^{n+1} b_{n-\ell} \frac{(-1)^{\ell}}{\ell!} w^{\ell} H_{\ell} + \left ( \sum_{\ell=0}^{n+1} b_{n-\ell} \frac{(-1)^{\ell+1}}{\ell!} w^{\ell} \right )\log w.
\end{equation}

We  prove by induction for $n\geq -1$ that we have $\mathbf{\hat S}_n(w)=\mathbf{S}_n(w) $, in three steps:

$\bullet$ For $n=-1$ we have $\mathbf{\hat r}_{-1}(w)=-\log (w)$ (because $b_{-1}=1$) and we have  proved the result in  Proposition \ref{prop:Formula_n=-1}.

$\bullet$ For $w=1$ we have $\mathbf{\hat S}_n(1)=\mathbf{S}_n(1)$. This is the integral formula from \cite{MPM4}.

$\bullet$ We check that we have
$$
\frac{d}{dw}\mathbf{\hat S}_0(w) = \mathbf{\hat S}_{-1}(w),
$$
and for $n\geq 1$,
$$
\frac{d}{dw}\mathbf{\hat S}_n(w) = n\mathbf{\hat S}_{n-1}(w)
$$
For $n\geq 1$, differentiating inside the integral and taking out of the integral the $n$-th term of the expansion of the Bernoulli generating function, we need to check that
\begin{equation}\label{eq:rec-rn}
-\frac{d}{dw} \mathbf{\hat r}_{n}(w)-\frac{b_n}{w}
= \mathbf{\hat r}_{n-1}(w) \ .
\end{equation}
We compute
\begin{align*}
-\frac{d}{dw} \mathbf{\hat r}_{n}(w)-\frac{b_n}{w}  &= \sum_{\ell=1}^{n+1} b_{n-\ell} \frac{(-1)^{\ell-1}}{(\ell-1)!} w^{\ell-1} H_\ell + \left ( \sum_{\ell=1}^{n+1} b_{n-\ell}\frac{(-1)^{\ell}}{(\ell-1)!} w^{\ell-1} \right ) \log w \\
&\ \ + \sum_{\ell=0}^{n+1} b_{n-\ell} \frac{(-1)^{\ell}}{\ell!} w^{\ell-1} -\frac{b_n}{w}\\
&= \sum_{\ell=1}^{n+1} b_{n-\ell} \frac{(-1)^{\ell-1}}{(\ell-1)!} w^{\ell-1} \left (H_{\ell-1}+\frac{1}{\ell} \right ) + \left ( \sum_{\ell=1}^{n+1} b_{n-\ell}\frac{(-1)^{\ell}}{(\ell-1)!} w^{\ell-1} \right ) \log w \\
&\ \ + \sum_{\ell=1}^{n+1} b_{n-\ell} \frac{(-1)^{\ell}}{\ell!} w^{\ell-1} \\
&= \sum_{\ell=1}^{n} b_{(n-1)-l} \frac{(-1)^{\ell}}{\ell!} w^{\ell} H_{\ell} +\sum_{\ell=1}^{n+1} b_{n-\ell} \frac{(-1)^{\ell-1}}{\ell!} w^{\ell-1} \\
&+ \left ( \sum_{\ell=0}^{n} b_{(n-1)-l}\frac{(-1)^{\ell+1}}{\ell!} w^{\ell} \right ) \log w + \sum_{\ell=1}^{n+1} b_{n-\ell} \frac{(-1)^{\ell}}{\ell!} w^{\ell-1} \\
&=\sum_{\ell=1}^{n} b_{(n-1)-l} \frac{(-1)^{\ell}}{\ell!} w^{\ell} H_{\ell} + \left ( \sum_{\ell=0}^{n} b_{(n-1)-l}\frac{(-1)^{\ell+1}}{\ell!} w^{\ell} \right ) \log w\\
&=\mathbf{\hat r}_{n-1}(w)\;.
\end{align*}

\subsubsection{Proof of Corollary \ref{cor:base_field}}

It is enough to prove

\begin{proposition}
For $n\geq 0$, $\mathbf{\hat r}_{n}(w)$ is an exponential period over the field
$
 \Q(t, e^{-t}, w, e^{-wt}) \ .
$
\end{proposition}

In view of the explicit expression of $\mathbf{\hat r}_{n}(w)$ this Proposition follows from the following Lemma:

\begin{lemma}
 For $n\geq 0$, $w^n \log w$ is an exponential period over the field
$
 \Q(t, e^{-t}, w, e^{-wt}) \ .
$
 
\end{lemma}

\begin{proof}
It is a direct  application of the Higher Frullani integral 
formula  (\cite{MPM1} and \cite{EM})  setting $s=w-1$.
\begin{theorem}[Higher Frullani Integrals] \label{thm:Frullani}
For $n\geq 0$, we have
\begin{align*}
(s+1)^n  \log (s+1)
&= \sum_{k=1}^n \binom{n}{k} (H_n-H_{n-k})s^k + \\
&\ \ \ +(-1)^{n+1} n! \int_0^{+\infty} \frac{1}{t^n}\left(e^{-st}- \sum_{k=0}^{n} \frac{(-s)^k t^k}{k!} \right)
\frac{e^{-t} dt}{t}\, .
\end{align*}

\end{theorem}

\end{proof}

\section{Relation to Hurwitz zeta function.}\label{sec:relation-to-Hurwitz}

In this section we prove Theorem \ref{thm:zetasn}.

For $\Re w >0$ and $\Re(s)>1$, we use  the integral representation of
the Hurwitz zeta function,
\[
\zeta_H(s,w)=\frac{1}{\Gamma(s)}\int_0^\infty t^{s-1} \frac{e^{-wt}}{1-e^{-t}}dt \ .
\]
For $n\geq 0$ and given $w$ with $\Re w>0$,
we have a meromorphic extension to $\Re s > -n-1$,
\begin{align*}
\zeta_H(s,w)
&=\frac{1}{\Gamma(s)} \int_0^\infty t^{s-1} \left(\frac{1}{1-e^{-t}} -\sum_{k=-1}^nb_k t^k \right) e^{-wt} \, dt+\sum_{k=-1}^n \frac{b_k}{w^{s+k}} \frac{\Gamma(s+k)}{\Gamma(s)} \\
&=\frac{1}{\Gamma(s)} \int_0^\infty t^{s-1} \left(\frac{1}{1-e^{-t}} -\sum_{k=-1}^nb_k t^k \right) e^{-wt} \, dt\\
& \ \ +\frac{b_{-1}}{w^{s-1}}\frac{1}{s-1}+\frac{b_0}{w^{s}}+\sum_{k=1}^n \frac{b_k}{w^{s+k}} (s+k-1)(s+k-2)\cdots s\;.
\end{align*}
The meromorphic extension is holomorphic at  $s = -n$ and we compute
its derivative with respect to the $s$-variable at this point.
First, we have for $1\leq k\leq n$,
 \[
 \frac{d}{ds} \Big((s+k-1)(s+k-2)\cdots s \Big )_{s=-n}=
 (-1)^{k+1}\frac{n!}{(n-k)!} (H_n-H_{n-k}).
 \]
and
\begin{align*}
\frac{\partial}{\partial s} \left (\frac{b_k}{w^{s+k}}
  \frac{\Gamma(s+k)}{\Gamma(s)}\right )_{{s=-n}}= & -b_k w^{n-k} (\log w) (-1)^{k} \frac{n!}{(n-k)!} \\
  & \ +b_kw^{n-k} (-1)^{k+1}\frac{n!}{(n-k)!} (H_n-H_{n-k}) \ .
\end{align*}
Moreover, since $\Gamma(-n)^{-1} =0$ and $\left ( \Gamma^{-1} \right )'(-n) =(-1)^n n!$, we have
\begin{align*}
&\frac{\partial}{\partial s} \left (\frac{1}{\Gamma(s)} \int_0^\infty t^{s-1} \left(\frac{1}{1-e^{-t}} -\sum_{k=-1}^nb_k t^k \right) e^{-wt} \, dt   \right )_{{s=-n}} \\
&= (-1)^n n!
 \int_0^\infty \frac{1}{t^n} \left(\frac{1}{1-e^{-t}} -\sum_{k=-1}^nb_k t^k \right) \frac{e^{-wt}}{t} \, dt
 \;.
\end{align*}
Therefore, we compute
 \begin{align*}
 \zeta_H'(-n,w)=&(-1)^n n!
 \int_0^\infty \frac{1}{t^n} \left(\frac{1}{1-e^{-t}} -\sum_{k=-1}^nb_k t^k \right) \frac{e^{-wt}}{t} \, dt\\
 &+\frac{b_{-1}}{n+1} w^{n+1}\log w-\frac{b_{-1}}{(n+1)^2}w^{n+1}-b_0 w^n\log w  \\
 &- \sum_{k=1}^n b_k w^{n-k} (\log w) (-1)^{k} \frac{n!}{(n-k)!}+\sum_{k=1}^n b_kw^{n-k} (-1)^{k+1}\frac{n!}{(n-k)!} (H_n-H_{n-k})\\
=&(-1)^n n!  \int_0^\infty \frac{1}{t^n} \left(\frac{1}{1-e^{-t}} -\sum_{k=-1}^nb_k t^k \right) \frac{e^{-wt}}{t} \, dt\\
 &-\frac{b_{-1}}{(n+1)^2}w^{n+1} + n!\sum_{\ell=0}^{n-1} b_{n-\ell}w^{\ell} (-1)^{n-\ell+1}\frac{1}{\ell !} (H_n-H_{\ell})\\
 &+(-1)^{n+1} n! \left(\sum_{\ell=0}^{n+1} b_{n-\ell}   \frac{(-1)^{\ell}}{\ell !}w^{\ell} \right)\log w \\
 &=(-1)^n n!  \int_0^\infty \frac{1}{t^n} \left(\frac{1}{1-e^{-t}} -\sum_{k=-1}^nb_k t^k \right) \frac{e^{-wt}}{t} \, dt\\
 &+(-1)^{n+1}n!\sum_{\ell=0}^{n+1} b_{n-\ell} \frac{(-1)^{\ell}}{\ell !} w^{\ell}(H_n-H_{\ell})\\
 &+(-1)^{n+1} n! \left(\sum_{\ell=0}^{n+1} b_{n-\ell}   \frac{(-1)^{\ell}}{\ell !}w^{\ell} \right)\log w\;.
\end{align*}
We recall the integral formula for $\mathbf{S}_n(w)$ from Theorem
\ref{thm:integral-rep-Sv},

\begin{equation*}
\mathbf{S}_n(w) = (-1)^{n+1} n! \left ( \int_0^{+\infty} \frac{1}{t^n} \left( \frac{1}{1-e^{-t}} - \sum_{k=-1}^{n} b_k t^{k} \right) \, \frac{e^{-wt} dt}{t}  +\mathbf{\hat r}_n(w)\right )
\end{equation*}
with
\[
\mathbf{\hat r}_{n}(w) = \sum_{\ell=1}^{n+1} b_{n-\ell} \frac{(-1)^{\ell}}{\ell!} w^{\ell} H_{\ell} + \left ( \sum_{\ell=0}^{n+1} b_{n-\ell} \frac{(-1)^{\ell+1}}{\ell!} w^{\ell} \right )\log w \;.
\]
From this formula it follows,

\begin{align*}
 \zeta_H'(-n,w)=&(-1)^n n!  \int_0^\infty \frac{1}{t^n} \left(\frac{1}{1-e^{-t}} -\sum_{k=-1}^nb_k t^k \right) \frac{e^{-wt}}{t} \, dt\\
 &+(-1)^{n+1}n! \left ( \sum_{\ell=0}^{n+1} b_{n-\ell}w^{\ell} (-1)^{\ell}\frac{1}{\ell !} \right ) H_n\\
 &+(-1)^{n} n! \, \mathbf{\hat r}_{n}(w) \\
 &=- \mathbf{S}_n(w) +(-1)^{n+1}n! \left ( \sum_{\ell=0}^{n+1} b_{n-\ell}w^{\ell} (-1)^{\ell}\frac{1}{\ell !} \right ) H_n \\
 &=- \mathbf{S}_n(w) +(-1)^{n+1}n! b_n(w) H_n \; .
\end{align*}
This concludes the proof of Theorem \ref{thm:zetasn}.

\section{Ramanujan-Candelpergher summation.}

\subsection{Background.}

Candelpergher has developed in his monograph \cite{Can}
a general theory of Ramanujan summation.
It is natural to ask about the relation with our ad hoc definition of Ramanujan
sum using the special form of the asymptotic expansions, which follows
closely Ramanujan's approach.

Ramanujan summation of the infinite series
$$
\sum_{k=1}^{+\infty} a_k
$$
is for series where the terms $a_k$ are given by $a_k=f(k)$ where $f$ is a function with good analytic properties. A holomorphic function $f(s)$ defined in the half plane
$\C_+(\sigma_0)=\{\Re s >\sigma_0\}$ for some $\sigma_0\in \RR$, $\sigma_0<1$, is of exponential type  $\alpha>0$,
if there exists  $\beta <\alpha$ and a constant $C>0$ such that for $s\in \C_+(\sigma_0)$,
$$
|f(s)|\leq C e^{\beta |s|} \ .
$$
The space of such functions is the vector space $\mathcal O^\alpha$. The classical
Theorem of Carlson (see \cite{Ca} and \cite{Can}) proves that if such a function $f$ exists then it
is unique in the space $\mathcal O^\pi$.
This is the main assumption in Candelpergher's procedure. Hence we are
able to sum only series where the general term $a_k$ grows at most exponentially with type $\alpha <\pi$.

Next, we find an interpolation function $\varphi$ of the finite sums,
$$
\varphi(s)=\sum_{k=1}^s a_k \;.
$$
that is, the function $\varphi$ is a solution of the difference equation
$$
\varphi(s)-\varphi (s-1) =a_s =f(s),
$$
such that $\varphi(0)=0$.
In general there is no unique solution to this difference equation. The solutions are
unique up to the addition of a $\Z$-periodic function vanishing at $0$.
In his notebooks, Ramanujan
introduces the  function $\varphi$ in his formulas and
computations, assuming that
$\varphi$ is a real interpolation, but he never considers
the latitude of the choice of $\varphi$, nor any complex extension of $\varphi$.
This is so because in practice, there is always a ``natural choice''. Indeed
in most cases the sequence $(a_k)$ is directly defined by a function $f$
in the space $\mathcal O^\pi$.

Difference equations were extensively studied by N\"orlund in \cite{No1} and
\cite{No2}. N\"orlund found conditions for the uniqueness of the solution when
$f$ extends to a complex neighborhood of the positive
real line $\RR_+$ and satisfies an exponential growth condition.

When $f\in \mathcal O^\pi$ there is a unique solution $\varphi\in \mathcal O^\pi$
of the difference
equation.  Again the uniqueness relies on Carlson's
Theorem. Candelpergher uses this function $\varphi$ to define the Ramanujan sum as
$$
\sum_{k\geq 1}^{\mathcal R} a_k =\int_0^1 \varphi(s) \, ds
$$
which has the expected good properties.

The case of interest for us is
the series giving $\mathbf S_n(w)$, with $\Re w>0$, namely
$$
\sum_{k=1}^{+\infty} (k+w)^n \log (k+w)
$$
and we want to understand the relation between
$\mathbf S_n(w) =S_{\mathcal R}((k+w)^n \log (k+w))$, given by the
constant term
of the asymptotic expansion, and Candelpergher's sum
$$
\sum_{k\geq 1}^{\mathcal R} (k+w)^n \log (k+w) \ .
$$
We can give explicit formulas for the function $\varphi$.
We have two approaches, via Hurwitz zeta function and via Frullani integrals.

\subsection{Hurwitz zeta function approach.}

For $\Re s >1$ and $\Re w >0$, the Hurwitz zeta function $\zeta_H$ is defined
by the converging series
$$
\zeta_H (s,w) =\sum_{k=0}^{+\infty} (k+w)^{-s} \ .
$$
(where $(k+w)^{-s} =\exp (-s\log(k+w))$ using the principal branch of the logarithm).
Since the sum is absolutely convergent, the function is holomorphic in the
two complex variables $(s,w)$. Therefore, we have for $\Re s >1$ and $\Re w >0$,
$$
\frac{\partial}{\partial s} \zeta_H (s,w)  =\zeta'_H (s,w) =  -\sum_{k=0}^{+\infty} (k+w)^{-s} \log(k+w)\,.
$$
Given $w$, the Hurwitz zeta function has a meromorphic extension in the $s$-variable
to the whole complex plane $s\in \C$, and the previous equation holds for
$(s,w)\in \C\times \C_+$. Moreover, we have, for $m\geq 1$,
$$
\sum_{k=1}^{m-1} (k+w)^{-s}=\zeta_H(s,w)-\zeta_H(s, m+w)-w^{-s},
$$
and differentiating with respect to the $s$-variable,
$$
\sum_{k=1}^{m} (k+w)^{-s}\log (k+w)=\zeta_H'(s, m+w)-\zeta_H'(s,w)+ (m+w)^{-s}\log (m+w)-w^{-s}\log w,
$$
specializing at $s=-n$ we get
$$
\varphi (m) = \zeta_H'(-n, w+m)-\zeta_H'(-n,w)+ (m+w)^{n}\log (m+w)-w^{n}\log w
$$
We check that $\varphi(0)=0$. We have that $\varphi$ is
holomorphic for $\Re m > -\Re w$. We need also to check that
$\varphi\in \mathcal O^\pi$.

\begin{lemma}
For fixed $s_0\in \C-\{1\}$ we have $\zeta_H(s_0,w)\in \mathcal O^\pi_w$.
\end{lemma}
\begin{proof}
Take $n\geq 1$ large enough so that $\Re s_0 > -n-1$ and use the meromorphic
extension from section \ref{sec:relation-to-Hurwitz}
\begin{align*}
\zeta_H(s_0,w)
&=\frac{1}{\Gamma(s_0)} \int_0^\infty t^{s_0-1} \left(\frac{1}{1-e^{-t}} -\sum_{k=-1}^nb_k t^k \right) e^{-wt} \, dt\\
& \ \ +\frac{b_{-1}}{w^{s_0-1}}\frac{1}{s_0-1}+\frac{b_0}{w^{s_0}}+\sum_{k=1}^n \frac{b_k}{w^{s_0+k}} (s_0+k-1)(s_0+k-2)\cdots s_0\;.
\end{align*}
For $\Re w\geq \sigma_0>0$, the integral is uniformly bounded in $w$ by
$$
\int_0^\infty t^{\Re s_0-1} \left |\frac{1}{1-e^{-t}} -\sum_{k=-1}^nb_k t^k \right | e^{-\sigma_0t} \, dt,
$$
and the rest of the expression has polynomial growth in $|w|$, hence for any $0<\eps<\pi$, $\zeta_H(s_0,w)\in \mathcal O^\eps_w \subset \mathcal O^\pi_w$
\end{proof}

\begin{proposition}
Let
$$
\varphi (x) = \zeta_H'(-n, w+x)-\zeta_H'(-n,w)+ (x+w)^{n}\log (x+w)-w^{n}\log w,
$$
then we have that $\varphi$ is holomorphic for $x\in \C$ with $\Re x > -\Re w$, and
$\varphi \in \mathcal O^\pi$.
\end{proposition}

\begin{proof}
Using the Lemma for $s_0=-n$ we get that $m\mapsto\zeta_H'(-n,w+m) \in \mathcal O^\pi$
(the space $\mathcal O^\pi$ is invariant by translations with positive real part). Also it is clear that $m\mapsto (m+w)^n \log(m+w) \in \mathcal O^\pi$ since it has
polynomial growth. Finally $\mathcal O^\pi$ is a vector space and therefore
$\varphi \in \mathcal O^\pi$.
\end{proof}

Now, we conclude:\begin{theorem} \label{thm:FormulaRC}
Let $n\geq 0$, we have
$$
\sum_{k \ge 1}^{\mathcal{R}} (k+w)^n\log (k+w)=-\frac{(w+1)^{n+1}}{(n+1)^2}  + \frac{(w+1)^{n+1}}{n+1}\log(w+1)
-\zeta_H'(-n,w)-w^{n}\log w\;.
$$
\end{theorem}

\begin{lemma} Let $n\geq 0$, we have
$$
\int_0^1 \zeta_H'(-n,x+w)\, dx= -\frac{w^{n+1}}{(n+1)^2}  + \frac{w^{n+1}}{n+1}\log(w)\,.
$$
\end{lemma}

\begin{proof}
We have for $\Re s>1$,
\begin{align*}
\int_0^1\zeta_H(s,x+w) \, dx &= \sum_{k=0}^\infty\int_0^1 \frac{1}{(k+x+w)^s}\, dx \\
&= \int_0^\infty \frac{1}{(x+w)^s} \, dx \\
&= \left[\frac{1}{-s+1}(x+w)^{-s+1}\right]_0^\infty \\
&= \frac{1}{s-1}\frac{1}{w^{s-1}}\,.
\end{align*}
This integral depends meromorphically on the parameter $s\in \C$, hence the equality
is valid for all $s\in \C$ integrating the meromorphic extension of the integrand.
We can take the partial derivative $\partial/\partial s$ and we get
$$
\int_0^1 \zeta'_H(s,x+w) \, dx = -\frac{1}{(s-1)^2} \frac{1}{w^{s-1}}
-\frac{1}{s-1} \frac{1}{w^{s-1}} \log w \, .
$$
Making $s=-n$ we get the result.
\end{proof}

\begin{lemma}\label{lem:simple_integral}
We have for all $n\geq 0$
$$
\int_0^1 (x+w)^n\log(x+w) \, dx = \frac{(w+1)^{n+1}}{n+1} \log(w+1)
-\frac{w^{n+1}}{n+1} \log(w) -\frac{(w+1)^{n+1}}{(n+1)^2} + \frac{w^{n+1}}{(n+1)^2}\,.
$$
\end{lemma}

\begin{proof}
Let $n\geq 0$, we have
\begin{align*}
\int_0^1 (x+w)^n \log(x+w) \, dx &= \int_w^{w+1} u^n \log u \, du \\
&=\left [ \frac{u^{n+1}}{n+1} \log u \right ]_w^{w+1} -\frac{1}{n+1} \int_w^{w+1} u^n \, du,
\end{align*}
and the result follows.
\end{proof}
\begin{proof}[Proof of Theorem \ref{thm:FormulaRC}]
We have using both Lemmas,
\begin{align*}
&\sum_{k \ge 1}^{\mathcal{R}} (k+w)^n\log (k+w) = \int_0^1 \varphi(x) \, dx \\
&=\int_0^1 \zeta_H'(-n,x+w) \, dx + \int_0^1 (x+w)^n\log(x+w) \, dx
-\zeta'(-n,w)-w^{n}\log w \\
&=-\frac{(w+1)^{n+1}}{(n+1)^2}  + \frac{(w+1)^{n+1}}{n+1}\log(w+1)
-\zeta'(-n,w)-w^{n}\log w\,.
\end{align*}
\end{proof}

\begin{corollary}\label{cor:165_1}
Let $n\geq 0$, we have
\[
\begin{split}
\sum_{k \ge 1}^{\mathcal{R}} (k+w)^n\log (k+w)=& \mathbf{S}_n(w) +(-1)^{n} n! b_n(w) H_n
-\frac{(w+1)^{n+1}}{(n+1)^2} \\
& + \frac{(w+1)^{n+1}}{n+1}\log(w+1)
-w^n\log w\,.
\end{split}
\]
\end{corollary}

\subsection{Frullani integrals approach.}

As in \cite{MPM4}, using Higher Frullani integrals, we can give an integral formula for the finite sums,
$$
\varphi(s)=\sum_{k=1}^s (k+w)^n \log(k+w).
$$
Higher Frullani integrals, Theorem  \ref{thm:Frullani}, provide an integral expression for the functions $s^n\log s$. We have

\begin{align*}
&\varphi(s)= \sum_{k=0}^{s-1}(k+w)^n \log(k+w)+(s+w)^n\log(s+w)-w^n\log w \\
&= w^n\sum_{k=0}^{s-1} \left(\frac{k}{w}+1\right)^n \log \left(\frac{k}{w}+1\right)
+  \left (\sum_{k=0}^{s-1}(k+w)^n \right )\log w +(s+w)^n\log(s+w)-w^n\log w \\
&= w^n\sum_{k=0}^{s-1} \left(\frac{k}{w}+1\right)^n \log \left(\frac{k}{w}+1\right)
+  \left (\sum_{k=1}^{s-1}(k+w)^n \right )\log w +(s+w)^n\log(s+w) \\
&= w^n\sum_{k=0}^{s-1} \left(\frac{k}{w}+1\right)^n \log \left(\frac{k}{w}+1\right)
+ \log w \sum_{j=0}^{n} \binom{n}{j} w^{n-j} \left (\sum_{k =1}^{s-1} k^j\right )
+(s+w)^n\log(s+w) \\
&=\sum_{k=1}^{n} \binom{n}{k} w^{n-k}\left(H_n - H_{n-k}\right)
\left (\sum_{\ell=0}^{s-1} \ell^k  \right ) \\
&+ (-1)^{n+1} n! \int_0^{+\infty} \frac{1}{t^{n}} \left( \frac{1-e^{-st}}{1-e^{-t}} -\sum_{k=0}^{n} \frac{(-t)^k}{k!} \left (\sum_{\ell={0}}^{s-1} \ell^k \right )\right) \frac{e^{-wt}}{t} \, dt\\
&+ \log w \sum_{j=0}^{n} \binom{n}{j} w^{n-j} \left (\sum_{\ell =1}^{s-1} \ell^j\right )
+(s+w)^n\log(s+w),
\end{align*}
where in the last step we have used
$$
\sum_{\ell=0}^{s-1} e^{-\ell t} =\frac{1-e^{-st}}{1-e^{-t}},
$$
and made the change of variable $t\mapsto wt$ in the integral.
Observe that
 $$
 \sum_{\ell={0}}^{s-1} \ell^k =(-1)^{k} k! (b_k -b_k(s)),
 $$
is the Faulhaber polynomial, hence we have computed  an explicit integral expression
for $\varphi$.
\begin{proposition}\label{prop:27_1} We have
 \begin{align*}
\vf(s)&= \sum_{k=1}^{n} \binom{n}{k} w^{n-k}\left(H_n - H_{n-k}\right)
\left (\sum_{\ell=0}^{s-1} \ell^k  \right ) \\
&+ (-1)^{n+1} n! \int_0^{+\infty} \frac{1}{t^{n}} \left( \frac{1-e^{-st}}{1-e^{-t}} -\sum_{k=0}^{n} \frac{(-t)^k}{k!} \left (\sum_{\ell={0}}^{s-1} \ell^k \right )\right) \frac{e^{-wt}}{t} \, dt\\
&+ \log w \sum_{j=0}^{n} \binom{n}{j} w^{n-j} \left (\sum_{\ell =1}^{s-1} \ell^j\right )
+(s+w)^n\log(s+w)  \, .
\end{align*}

\end{proposition}

From this explicit integral formula we obtain directly,

\begin{proposition}
We have $\varphi(0)=0$ and $\varphi \in \mathcal O^\pi$.
\end{proposition}

\begin{proof}
 Again observe that
 $$
 \sum_{\ell={0}}^{s-1} \ell^k =(-1)^{k} k! (b_k -b_k(s))
 $$
 is the Faulhaber polynomial (thus we have polynomial growth on $s$),
 the  integral part is uniformly bounded, and the last term  is polynomially bounded.
\end{proof}

Now we are in Candelpergher's setup and we compute Ramanujan-Candelpergher sum as
$$
\sum_{k \ge 1}^{\mathcal{R}} (k+w)^n\log (k+w)=\int_0^1 \varphi(x) \, dx \ .
$$
We have two simple Lemmas.
\begin{lemma}
We have
$$
\int_0^1 \frac{1-e^{-st}}{1-e^{-t}} \, ds = \frac{1}{1-e^{-t}} -\frac1t \, .
$$
\end{lemma}
Observe that
$$
\frac{1-e^{-st}}{1-e^{-t}} = \frac{1}{1-e^{-t}}-\frac{e^{-st}}{1-e^{-t}}
= \frac{1}{1-e^{-t}}- \sum_{k=0}^{+\infty} b_k(s) t^k,
$$
so we have  $\int_0^1 b_k(s)\,  ds=0$ for $k\geq 0$. It follows:
\begin{lemma}\label{lemma:27_1}
We have for $k\geq 0$,
$$
\int_0^1 \left (\sum_{j=0}^{s-1} j^k \right ) ds =  (-1)^{k} k! \, b_k \ .
$$
\end{lemma}

Therefore, we finally get, using the integral from Lemma \ref{lem:simple_integral},

\begin{theorem}\label{thm:4.12}
 We have
\begin{align*}
\sum_{k \ge 1}^{\mathcal{R}} & (k+w)^n\log (k+w) =\sum_{k=1}^{n} \binom{n}{k} w^{n-k}\left(H_n - H_{n-k}\right)  (-1)^{k} k! \, b_k \\
&+ (-1)^{n+1} n! \int_0^{+\infty} \frac{1}{t^n} \left( \frac{1}{1-e^{-t}} -\sum_{k=-1}^n t^k b_k  \right )\frac{e^{-wt}}{t} \, dt \\
&+ (-1)^n n! \log w \sum_{j=0}^{n}   \frac{(-1)^j}{j!} w^{j}b_{n-j}\\
&+\frac{(w+1)^{n+1}}{n+1} \log(w+1)
-\frac{w^{n+1}}{n+1} \log(w) -\frac{(w+1)^{n+1}}{(n+1)^2} + \frac{w^{n+1}}{(n+1)^2}\\
&-w^n\log w\, .
\end{align*}
\end{theorem}

\begin{proof}
We compute $\int_0^1\vf(s)ds$. Using Proposition \ref{prop:27_1} and Lemma \ref{lemma:27_1} we compute
 \begin{align*}
&\int_0^1  \vf(s)ds = \sum_{k=1}^{n} \binom{n}{k} w^{n-k}\left(H_n - H_{n-k}\right)
(-1)^k k! b_k \\
&+ (-1)^{n+1} n! \int_0^{+\infty} \frac{1}{t^{n}} \left( \frac{1}{1-e^{-t}} -\frac{1}{t}-\sum_{k=0}^{n} \frac{(-t)^k}{k!}(-1)^k k! b_k \right) \frac{e^{-wt}}{t} \, dt\\
&+ \log w \sum_{j=0}^{n} \binom{n}{j} w^{n-j} \int_0^1\left (\sum_{\ell =1}^{s-1} \ell^j\right )ds +\int_0^1(s+w)^n\log(s+w)ds  \\
&= \sum_{k=1}^{n} \binom{n}{k} w^{n-k}\left(H_n - H_{n-k}\right){(-1)^k k! b_k} + (-1)^{n+1} n! \int_0^{+\infty} \frac{1}{t^{n}} \left( \frac{1}{1-e^{-t}} -\sum_{k=-1}^{n} {t^k}b_k \right) \frac{e^{-wt}}{t} \, dt\\
&+ \log w \sum_{j=1}^{n} \binom{n}{j} w^{n-j} \int_0^1\left (\sum_{\ell =1}^{s-1} \ell^j\right )ds+ \log w \; w^n \int_0^1\left (\sum_{\ell =1}^{s-1} \ell^0\right )ds +\int_0^1(s+w)^n\log(s+w)ds  \\
&= \sum_{k=1}^{n} \binom{n}{k} w^{n-k}\left(H_n - H_{n-k}\right){(-1)^k k! b_k} + (-1)^{n+1} n! \int_0^{+\infty} \frac{1}{t^{n}} \left( \frac{1}{1-e^{-t}} -\sum_{k=-1}^{n} {t^k}b_k \right) \frac{e^{-wt}}{t} \, dt\\
&+ \log w {\sum_{j=1}^{n}} \binom{n}{j} w^{n-j} (-1)^j j! b_j+ \log w \; w^n {\int_0^1\left (s-1\right )ds} +\int_0^1(s+w)^n\log(s+w)ds  \\
&= \sum_{k=1}^{n} \binom{n}{k} w^{n-k}\left(H_n - H_{n-k}\right){(-1)^k k! b_k} + (-1)^{n+1} n! \int_0^{+\infty} \frac{1}{t^{n}} \left( \frac{1}{1-e^{-t}} -\sum_{k=-1}^{n} {t^k}b_k \right) \frac{e^{-wt}}{t} \, dt\\
&+ \log w \sum_{j=1}^{n} \binom{n}{j} w^{n-j} \int_0^1\left (\sum_{\ell =1}^{s-1} \ell^j\right )ds+ \log w \; w^n \int_0^1\left (\sum_{\ell =1}^{s-1} \ell^0\right )ds +\int_0^1(s+w)^n\log(s+w)ds  \\
&= \sum_{k=1}^{n} \binom{n}{k} w^{n-k}\left(H_n - H_{n-k}\right){(-1)^k k! b_k} + (-1)^{n+1} n! \int_0^{+\infty} \frac{1}{t^{n}} \left( \frac{1}{1-e^{-t}} -\sum_{k={-1}}^{n} {t^k}b_k \right) \frac{e^{-wt}}{t} \, dt\\
&+ (-1)^nn! \log w {\sum_{j=0}^{n}} (-1)^j  \frac{w^{j}}{ j! }b_{n-j}-w^n\log w +\frac{(w+1)^{n+1}}{n+1} \log(w+1)
-\frac{w^{n+1}}{n+1} \log(w) \\
&-\frac{(w+1)^{n+1}}{(n+1)^2}  + \frac{w^{n+1}}{(n+1)^2}
\end{align*}
which gives the result.

\end{proof}

Using the integral formula for $\mathbf{S}_n(w)$ from Theorem \ref{thm:integral-rep-Sv}, we get the relation of our Ramanujan summation with Candelpergher summation.
\begin{theorem}\label{thm:4.13}
We have
$$
\sum_{k \ge 1}^{\mathcal{R}} (k+w)^n\log (k+w)= \mathbf{S}_n(w) +(-1)^{n} n! b_n(w) H_n
-\frac{(w+1)^{n+1}}{(n+1)^2}  + \frac{(w+1)^{n+1}}{n+1}\log(w+1)-w^n\log w\,.
$$

\end{theorem}

\part{}

\section{Other applications.}

\subsection{On determinants of Laplacians.}
Theorem \ref{thm:log-determinants_are_periods} is a direct application of our previous results, combined
with Kumagai's formulas for  the regularized determinant $\det \Delta$ of $M=\mathbb{S}^d$. From Kumagai's formulas \cite{Kum}, there are explicit integers
 $(\alpha_{d,j})_{0\leq j\leq d}$ and $(\tau_{d,j})_{0\leq j\leq d}$, and a rational
 number $\gamma_d$, such that
 $$
 \log \det\nolimits \Delta_{\mathbb{S}^d} = \gamma_d +\sum_{k=1}^d \alpha_{d,k} \log k +  \sum_{k=0}^{d} \tau_{d,k} \zeta'(-k),
 $$
Using Frullani integral
$$
\log (s+1) =\int_0^{+\infty} (1-e^{-st}) \, \frac{e^{-t} dt}{t}
$$
we have that for any integer $n\geq 2$, $\log (n)$ is an exponential period over the base
field $\Q(t,e^{-t})$.
Thus, we get Theorem \ref{thm:log-determinants_are_periods} as a direct
Corollary of the integral representation
of Stirling-Ramanujan constants from \cite{MPM4} and
its relation with $\zeta'(-n)$ given by  Corollary \ref{cor:Adamchik_formula}.

We prove the particular cases in Proposition \ref{prop:examples}.
From \cite{MPM4}, we recall the following explicit integral representations for the classical Stirling-Ramanujan constants $S_k$ for $k=0,1,2$
:
\begin{align*}
S_0 &=\frac{\log(2\pi)}{2}
=
-\int_0^{+\infty}
\left(
\frac{1}{1-e^{-t}}
-\frac1t
-\frac12
-t
\right)
\frac{e^{-t}}{t}\,dt
\\
S_1 &=
\int_0^{+\infty}
\frac1t
\left(
\frac{1}{1-e^{-t}}
-\frac1t
-\frac12
-\frac{t}{12}
+\frac{t^2}{4}
\right)
\frac{e^{-t}}{t}\,dt
\\
S_2 &=-2
\int_0^{+\infty}
\frac{1}{t^2}
\left(
\frac{1}{1-e^{-t}}
-\frac1t
-\frac12
-\frac{t}{12}
-\frac{t^3}{72}
\right)
\frac{e^{-t}}{t}\,dt \ .
\end{align*}

On the other hand, the regularized determinants of the spheres
$\mathbb S^d$ for $d=1,2,3$ are computed in
\cite{Kum}. Combining these formulas with the
period integral representations yields the examples
for $d=1,2,3$  from Proposition \ref{prop:examples}.
More precisely, for $d=1$, we have the classical result
$$
\det \Delta_{\mathbb S^1} = 4\pi^2
$$
which happens to coincide with the zeta-regularized product of the prime numbers (see \cite{MPM0}).
Therefore, we have $\log \det \Delta_{\mathbb S^1}=4S_0,$ and
\[
\log \det \Delta_{\mathbb S^1}
=
-4
\int_0^{+\infty}
\left(
\frac{1}{1-e^{-t}}
-\frac1t
-\frac12
-t
\right)
\frac{e^{-t}}{t}\,dt.
\]
For $d=2$,  we have $ \log \det \Delta_{\mathbb S^2}=
-4\zeta'(-1)+\frac12.$
 Since $
-4\zeta'(-1)+\frac12=4S_1+\frac16,$
we obtain
\[
\log \det \Delta_{\mathbb S^2}
=
4S_1+\frac16,
\]
and therefore
\[
\log \det \Delta_{\mathbb S^2}
=
4
\int_0^{+\infty}
\frac1t
\left(
\frac{1}{1-e^{-t}}
-\frac1t
-\frac12
-\frac{t}{12}
+\frac{7}{24}t^2
\right)
\frac{e^{-t}}{t}\,dt\;.
\]

For $d=3$, we have
\[
\log \det \Delta_{\mathbb S^3}
=
-2\zeta'(-2)-2\zeta'(0)-\log 2.
\]
We note that the formula given in \cite[(iii) of Corollary]{Kum} contains a typo. The correct formula is
\[
\det \Delta_{\mathbb S^3}
=
\pi \exp\left(\frac{\zeta(3)}{2\pi^2}\right).
\]
Using the identities
\[
\zeta'(0)=-\frac{1}{2}\log(2\pi)
\qquad\text{and}\qquad
\zeta'(-2)=-\frac{\zeta(3)}{4\pi^2},
\]
we recover the aforementioned expression for $\log \det \Delta_{\mathbb S^3}$.

Since
\[
-\zeta'(-2)=S_2
\qquad\text{and}\qquad
-\zeta'(0)=S_0,
\]
we deduce that
\[
\log \det \Delta_{\mathbb S^3}
=
2S_2+2S_0-\log 2.
\]
Using the integral representations of $S_2$, $S_0$ and that
 $$-\log 2% = -2\int_0^\infty \frac{1-e^{-t}}{2}\,\frac{e^{-t}}{t}\,dt
=
 -2\int_0^\infty\left(\frac12-\frac{e^{-t}}{2}\right)\frac{e^{-t}}{t}\,dt,$$
 we therefore obtain
\[
\begin{aligned}
\log \det \Delta_{\mathbb S^3}
=
-2\int_0^\infty
\left(
\left(1+\frac{2}{t^2}\right)\frac{1}{1-e^{-t}}
-\frac{2}{t^3}
-\frac{7}{6t}
-\frac{1}{t^2}
-\frac{37}{36}t
-\frac{1}{2}e^{-t}
\right)
\frac{e^{-t}\,dt}{t}\;.
\end{aligned}
\]

\medskip

\subsubsection*{Determinant of Laplacians of lens spaces}

Let $\mathbb{C}^{n+1}$ be equipped with the standard flat Kähler metric
\[
ds^2=\sum_{k=0}^n \lvert dz_k\rvert^2.
\]
For a positive integer $q$ and integers $p_0,\ldots,p_n$ coprime to $q$, let
\[
\gamma=e^{2\pi i/q}.
\]
The isometry $g\in U(n+1)$ defined by
\[
g(z_0,\ldots,z_n)
=
(\gamma^{p_0}z_0,\ldots,\gamma^{p_n}z_n)
\]
generates a cyclic subgroup
\[
G=\{g^k\}_{k=0}^{q-1}.
\]
The quotient
\[
L(q;p_0,\ldots,p_n)
=
\mathbb{S}^{2n+1}/G
\]
of the unit sphere $\mathbb{S}^{2n+1}\subset\mathbb{C}^{n+1}$ by the free $G$-action is a Riemannian manifold called a lens space, with metric induced by the local isometry
\[
\pi:\mathbb{S}^{2n+1}\longrightarrow\mathbb{S}^{2n+1}/G.
\]

The spectrum of the Laplacian on $L(q;p_0,\ldots,p_n)$ is well studied through the generating function associated with the spectrum of a $(2n+1)$-dimensional lens space of constant curvature $1$, namely
\[
F(z)
=
\sum_{k=0}^{\infty}
\bigl(\dim E_{k(k+2n)}\bigr)z^k,
\]
where $E_{k(k+2n)}$ denotes the eigenspace corresponding to the eigenvalue
$k(k+2n)$. This generating function is a rational function with rational coefficients; see \cite{Ike,Ike-Yam}.

We are therefore in a position to apply \cite[Theorem~1.2]{Ha}. It remains to show that $f'(0)$ (or $g'(0)$), in the notation of \cite[Theorem~1.2]{Ha}, is an exponential period over $\mathbb{Q}(t,e^{-t})$. In our setting,
\[
f(s)
=
\sum_{k=1}^{\infty}
\frac{\dim E_{k(k+2n)}}{k^s}.
\]
Using the generating function $F$ and \cite[(2.3)]{Ike}, we can write
\[
f(s)
=
\frac{1}{\Gamma(s)}
\int_0^\infty
t^{s-1}\bigl(F(e^{-t})-1\bigr)\,dt.
\]

Note that $F(0)=1$; see, for instance, \cite[p.~310]{Ike}. A straightforward generalization of the proof of the integral representation for $\zeta_H'(0,w)$ given above yields
\[
f'(0)
=
\int_0^\infty
\bigl(F(e^{-t})e^t-e^t\bigr)
\frac{e^{-t}\,dt}{t}.
\]

Since $F$ is a rational function with rational coefficients, it follows that $f'(0)$ is an exponential period over $\mathbb{Q}(t,e^{-t})$. Moreover, the remaining terms in \cite[Theorem~1.2]{Ha} are rational numbers. This proves Theorem~\ref{thm:lens}.

\subsection{On the special values $\zeta(2n+1)/\pi^{2n}$.}

Using the functional equation we can relate the special
values $\zeta'(-2n)$ to the
values taken by Riemann zeta function at positive integer values.
We have the following elementary Lemma.

\begin{lemma}
We have for $n\geq 1$,
\begin{align*}
 \zeta'(-2n) &= (-1)^n \frac{(2n)!}{2(2\pi)^{2n}} \zeta(2n+1)\,. \end{align*}
\end{lemma}

\begin{proof}
 We take the derivative of the functional equation
 $$
 \zeta(1-s) = 2 (2\pi)^{-s} \cos\left (\frac{\pi s}{2} \right ) \Gamma(s) \zeta(s),
 $$
and we get
\begin{align*}
\zeta'(1-s) =& +2 \log (2\pi) (2\pi)^{-s} \cos\left (\frac{\pi s}{2} \right ) \Gamma(s) \zeta(s)\\
&+2 (2\pi)^{-s} \left ( \frac{\pi}{2} \right ) \sin\left (\frac{\pi s}{2} \right ) \Gamma(s) \zeta(s) \\
&-2 (2\pi)^{-s} \cos\left (\frac{\pi s}{2} \right ) \Gamma'(s) \zeta(s) \\
&-2 (2\pi)^{-s} \cos\left (\frac{\pi s}{2} \right ) \Gamma(s) \zeta'(s)\,.
\end{align*}
Making $s=2n+1$ with $n\geq 1$, we have
$$
\zeta'(-2n) = \pi (2\pi)^{-(2n+1)} (-1)^n (2n)! \zeta(2n+1) = (-1)^n \frac{(2n)!}{2(2\pi)^{2n}} \zeta(2n+1)\,.
$$

\end{proof}

Therefore, we have
$$
\frac{\zeta(2n+1)}{\pi^{2n}} =(-1)^n \frac{2^{2n+1}}{(2n)!}  \zeta'(-2n),
$$

and using Corollary \ref{cor:Adamchik_formula}, $\mathbf{S}_{2n}(1) =-\zeta'(-2n) +\frac{B_{2n+1}}{2n+1} H_{2n}$, we get
$$
\frac{\zeta(2n+1)}{\pi^{2n}} = (-1)^{n+1} \frac{2^{2n+1}}{(2n)!} \mathbf{S}_{2n}(1) \, .
$$
Now, using Theorem \ref{thm:integral-rep-Sv},  this gives the exponential period representation for the values $\frac{\zeta(2n+1)}{\pi^{2n}} $, that is given in Theorem \ref{thm:zeta(2n+1)}.

\part{}

\section{Twisted series}

Our primary goal is to study the asymptotic behaviour of the sums
\[
 \sum_{k=1}^{ms} \chi(k)\, P(k)\log Q(k),
\]
where $P,Q\in\C[s]$ and $\chi$ is a periodic function of period $m$. We denote by
\[
\hat{S}^{P,Q}(\chi)
\]
the constant term appearing in the corresponding asymptotic expansion.

\subsection{Preliminary reductions}

Let $\chi:\mathbb Z\to \mathbb C$ be a periodic function of period $m$. Given two polynomials
\[
P,Q\in\C[s],
\]
with $\deg Q\geq1$ and leading coefficient of $Q$ having positive real part, we consider divergent series of the form
\begin{equation}\label{eq:series1}
\sum_{k=1}^{ms}\chi(k)P(k)\log Q(k).
\end{equation}

The first reduction consists in assuming that $\chi(0)=0$. Indeed,
\[
\sum_{k=1}^{ms}\chi(k)P(k)\log Q(k)
=
\sum_{k=1}^{ms}(\chi(k)-\chi(0))P(k)\log Q(k)
+
\chi(0)\sum_{k=1}^{ms}P(k)\log Q(k),
\]
and the asymptotic expansion of the second sum has already been studied. Hence, in the sequel, we assume
\[
\chi(0)=0 \ .
\]

Since the asymptotic expansion is unaffected by the addition or removal of finitely many terms, we may further simplify the situation. Because the leading coefficient of $Q$ has positive real part, there exists $k_0\geq0$ such that
\[
\Re Q(k)>0,
\qquad k\geq k_0.
\]
Replacing $Q(s)$ by $Q(s+k_0)$, we may therefore assume that $\log Q(k)$ is defined using the principal branch of the logarithm for all $k\geq1$. Moreover, we may assume that every zero $\omega$ of $Q$ satisfies
\[
\Re \omega <1.
\]

Factoring
\[
Q(s)=a\prod_{\omega}(s-\omega)^{n_\omega},
\]
the study of \eqref{eq:series1} reduces to the analysis of the two classes of series
\begin{equation}\label{eq:series2}
\sum_{k=1}^{ms}\chi(k)P(k)
\end{equation}
and
\begin{equation}\label{eq:series3}
\sum_{k=1}^{ms}\chi(k)P(k)\log(k+w),
\end{equation}
where $\Re w>-1$ (with $w=-\omega$).

Expanding $P$ into monomials, the analysis of the series \eqref{eq:series2} follows directly from Faulhaber's formula, while the study of \eqref{eq:series3} reduces to the series
\begin{equation}\label{eq:series4}
\sum_{k=1}^{ms}\chi(k)\,k^n\log(k+w),
\end{equation}
where $n\geq0$ is an integer.

The asymptotic analysis of \eqref{eq:series4} is equivalent to that of
\begin{equation}\label{eq:series5}
\sum_{k=1}^{ms}\chi(k)\,(k+w)^n\log(k+w),
\end{equation}
since
\[
\sum_{k=1}^{ms}\chi(k)(k+w)^n\log(k+w)
=
\sum_{\ell=0}^{n}
\binom{n}{\ell}
w^{n-\ell}
\sum_{k=1}^{ms}\chi(k)\,k^\ell\log(k+w),
\]
and conversely,
\[
\sum_{k=1}^{ms}\chi(k)\,k^n\log(k+w)
=
\sum_{\ell=0}^{n}
\binom{n}{\ell}
(-w)^{\,n-\ell}
\sum_{k=1}^{ms}\chi(k)(k+w)^\ell\log(k+w).
\]

It is also convenient to shift the summation index and consider instead
\begin{equation}\label{eq:series6}
\sum_{k=0}^{ms-1}\chi(k)(k+w)^n\log(k+w).
\end{equation}
Indeed, since $\chi(0)=0$, we have
\[
\sum_{k=1}^{ms}\chi(k)(k+w)^n\log(k+w)
=
\sum_{k=0}^{ms-1}\chi(k)(k+w)^n\log(k+w).
\]

Partitioning the sum according to congruence classes modulo $m$, we obtain
\[
\begin{aligned}
\sum_{k=0}^{ms-1}\chi(k)(k+w)^n\log(k+w)
&=
\sum_{a=0}^{m-1}\chi(a)
\sum_{k=0}^{s-1}(a+km+w)^n\log(a+km+w)
\\
&=
m^n\log m
\sum_{a=0}^{m-1}\chi(a)
\sum_{k=0}^{s-1}
\left(k+\frac{a+w}{m}\right)^n
\\
&\quad
+
m^n
\sum_{a=0}^{m-1}\chi(a)
\sum_{k=0}^{s-1}
\left(k+\frac{a+w}{m}\right)^n
\log\left(k+\frac{a+w}{m}\right).
\end{aligned}
\]

Hence, the problem reduces to the study of sums of the form
\[
\sum_{k=0}^{s-1}(k+v)^n\log(k+v),
\qquad
v=\frac{a+w}{m}.
\]

As a consequence, the constant
\[
\hat{S}^{(k+w)^n,(k+w)}(\chi)
\]
admits the representation
\[
\hat{S}^{(k+w)^n,(k+w)}(\chi)
=
m^n
\sum_{a=0}^{m-1}
\chi(a)\,
\mathbf{S}_n\!\left(\frac{a+w}{m}\right).
\]

Our main result in this section is the following explicit integral representation for the twisted Stirling--Ramanujan constants
\[
\hat{S}^{(k+w)^n,(k+w)}(\chi).
\]

\begin{theorem}\label{thm:105_2}
Let $n\geq 0$ and let $w\in \mathbb C$ satisfy $\Re(w)>-1$. Then
\[
\begin{aligned}
\hat{S}^{(k+w)^n,(k+w)}(\chi)
&=
(-1)^{n+1}n!
\int_0^{+\infty}
\frac{1}{t^n}
\left(
\frac{F_{\chi}(-t,w-1)}{t}
-
\sum_{k=-1}^{n} b_{\chi,k}(w)t^k
\right)
\frac{e^{-t}\,dt}{t}
\\
&\quad-
(-1)^{n+1}n!
\sum_{j=0}^{n+1}
\frac{(-1)^{j+1}}{j!}
H_j\, b_{\chi,n-j}(w)
\\
&\quad+
(-1)^{n+1}n!\,
b_{\chi,n}(w+1)\log m.
\end{aligned}
\]

The twisted Bernoulli polynomials $b_{\chi,k}(w)$ are defined by the generating series
\[
\frac{1}{t}F_{\chi}(-t,w-1)
=
\sum_{k=-1}^{\infty}
b_{\chi,k}(w)t^k,
\]
and the twisted Bernoulli numbers are defined by
\[
b_{\chi,k}:=b_{\chi,k}(0)\,.
\]
\end{theorem}

The proof relies on the following combinatorial identity.

\begin{lemma}
For every integer $n\geq -1$,
\begin{equation}\label{eq:145_2}
\sum_{a=0}^{m-1} \chi(a)\sum_{\ell=0}^{n+1}
m^{n-\ell} b_{n-\ell}
\frac{(-1)^{\ell+1}}{\ell!}
\sum_{j=0}^{\ell}
\binom{\ell}{j}
H_{\ell-j}(a+w-1)^j
=
\sum_{\ell=0}^{n+1}
b_{\chi,n-\ell}(w)
H_\ell
\frac{(-1)^{\ell+1}}{\ell!}\,.
\end{equation}
\end{lemma}

\begin{proof}

We have
\[
\begin{aligned}
-e^{-(a+w-1)t}
\frac{1}{1-e^{-mt}}
\sum_{\ell=0}^{\infty}
\frac{H_\ell}{\ell!}(-t)^\ell
&=
\left(
\sum_{j=0}^{\infty}
\frac{(-1)^{j+1}}{j!}
(a+w-1)^j t^j
\right)
\left(
\sum_{k=-1}^{\infty}
m^k b_k t^k
\right)
\\
&\qquad\qquad\times
\left(
\sum_{\ell=0}^{\infty}
\frac{H_\ell}{\ell!}(-t)^\ell
\right).
\end{aligned}
\]

Expanding the product yields
\[
\begin{aligned}
-e^{-(a+w-1)t}
\frac{1}{1-e^{-mt}}
\sum_{\ell=0}^{\infty}&
\frac{H_\ell}{\ell!}(-t)^\ell
=
\sum_{j\geq0,\;k\geq-1,\;\ell\geq0}
\frac{(-1)^{j+\ell+1}}{j!\,\ell!}
(a+w-1)^j
m^k b_k H_\ell\,
t^{j+k+\ell}
\\
&=
\sum_{N=-1}^{\infty}
\left(
\sum_{j=0}^{N+1}
\frac{(-1)^{j+1}}{j!}
(a+w-1)^j
\sum_{\ell=0}^{N+1-j}
m^{N-j-\ell}
b_{N-j-\ell}
\frac{(-1)^\ell}{\ell!}
H_\ell
\right)
t^N,
\end{aligned}
\]
where we used $H_0=0$.

On the other hand,
\[
\begin{aligned}
&\sum_{\ell=0}^{n+1}
m^{n-\ell} b_{n-\ell}
\frac{(-1)^{\ell+1}}{\ell!}
\sum_{j=0}^{\ell}
\binom{\ell}{j}
H_{\ell-j}(a+w-1)^j
\\
&\qquad=
\sum_{j=0}^{n+1}
\sum_{\ell=j}^{n+1}
m^{n-\ell} b_{n-\ell}
\frac{(-1)^{\ell+1}}{j!(\ell-j)!}
H_{\ell-j}(a+w-1)^j
\\
&\qquad=
\sum_{j=0}^{n+1}
\frac{(-1)^{j+1}}{j!}
(a+w-1)^j
\sum_{\ell=0}^{n+1-j}
m^{n-j-\ell}
b_{n-j-\ell}
\frac{(-1)^\ell}{\ell!}
H_\ell.
\end{aligned}
\]

Multiplying by $\chi(a)$ and summing over $a$ gives
\[
\sum_{a=0}^{m-1} \chi(a)\sum_{\ell=0}^{n+1}
m^{n-\ell} b_{n-\ell}
\frac{(-1)^{\ell+1}}{\ell!}
\sum_{j=0}^{\ell}
\binom{\ell}{j}
H_{\ell-j}(a+w-1)^j
=
\sum_{\ell=0}^{n+1}
b_{\chi,n-\ell}(w)
H_\ell
\frac{(-1)^{\ell+1}}{\ell!},
\]
we have used
\[
\frac{1}{t}F_\chi(-t,w-1)
=
\sum_{n=-1}^{\infty} b_{\chi,n}(w)t^n.
\]
This concludes the proof of  the lemma.
\end{proof}

\begin{proof}[Proof of Theorem \ref{thm:105_2}]

From the integral representation obtained previously, we have
\[
\begin{aligned}
\hat{S}^{(k+w)^n,(k+w)}(\chi)
&=
(-1)^{n+1}n!
\int_0^{+\infty}
\frac{1}{t^n}
\Bigg(
\frac{F_{\chi}(-t,w-1)}{t}
\\
&\qquad\qquad\qquad\qquad
-
\sum_{k=-1}^{n}
m^k b_k t^k
\sum_{a=1}^{m}
\chi(a)e^{-(a+w-1)t}
\Bigg)
\frac{e^{-t}\,dt}{t}
\\
&\quad+
(-1)^{n+1}n!\,
m^n
\sum_{a=1}^{m}
\chi(a)\,
\hat{\mathbf r}_n\!\left(\frac{a+w}{m}\right).
\end{aligned}
\]

We decompose
\[
\begin{aligned}
\sum_{k=-1}^{n}
m^k b_k t^k e^{-(a+w-1)t}
&=
\sum_{k=-1}^{n}
m^k b_k t^k
\left(
e^{-(a+w-1)t}
-
\sum_{\ell=0}^{n-k}
\frac{(-(a+w-1))^\ell}{\ell!}t^\ell
\right)
\\
&\quad+
\sum_{k=-1}^{n}
m^k b_k t^k
\sum_{\ell=0}^{n-k}
\frac{(-(a+w-1))^\ell}{\ell!}t^\ell\,.
\end{aligned}
\]

Multiplying by $\chi(a)$, summing over $a$, and identifying coefficients with the Laurent expansion
\[
\frac{1}{t}F_{\chi}(-t,w-1)
=
\sum_{k=-1}^{\infty}
b_{\chi,k}(w)t^k,
\]
gives
\begin{equation}\label{eq:105_1}
\sum_{a=1}^{m}\chi(a)
\sum_{k=-1}^{n}
m^k b_k t^k
\sum_{\ell=0}^{n-k}
\frac{(-(a+w-1))^\ell}{\ell!}t^\ell
=
\sum_{k=-1}^{n}
b_{\chi,k}(w)t^k.
\end{equation}

Equivalently, for every $n\geq -1$,
\begin{equation}\label{eq:145_1}
b_{\chi,n}(w)
=
\sum_{a=1}^{m}\chi(a)
\sum_{\ell=0}^{n+1}
m^{\,n-\ell}
b_{n-\ell}
\frac{(-(a+w-1))^\ell}{\ell!}.
\end{equation}

Consequently, we obtain the identity
\[
\begin{aligned}
\sum_{k=-1}^{n} m^k b_k t^k
\sum_{a=1}^{m}\chi(a)e^{-(a+w-1)t}
&=
\sum_{a=1}^{m}\chi(a)
\sum_{k=-1}^{n} m^k b_k t^k
\left(
e^{-(a+w-1)t}
-
\sum_{\ell=0}^{n-k}
\frac{(-(a+w-1))^\ell}{\ell!}t^\ell
\right)
\\
&\quad+
\sum_{k=-1}^{n} b_{\chi,k}(w)t^k.
\end{aligned}
\]

\medskip

Substituting this decomposition into the integral and isolating the \(\chi\)-twisted singular contribution, we obtain
\[
\begin{aligned}
&\int_0^{+\infty}
\frac{1}{t^n}
\left(
\frac{F_{\chi}(-t,w-1)}{t}
-
\sum_{k=-1}^{n}
m^k b_k t^k
\sum_{a=1}^{m}\chi(a)e^{-(a+w-1)t}
\right)
\frac{e^{-t}\,dt}{t}
\\
&=
\int_0^{+\infty}
\frac{1}{t^n}
\left(
\frac{F_{\chi}(-t,w-1)}{t}
-
\sum_{k=-1}^{n}
b_{\chi,k}(w)t^k
\right)
\frac{e^{-t}\,dt}{t}
\\
&\quad-
\sum_{a=1}^{m}\chi(a)
\sum_{k=-1}^{n} m^k b_k
\int_0^{+\infty}
t^{k-n}
\left(
e^{-(a+w-1)t}
-
\sum_{\ell=0}^{n-k}
\frac{(-(a+w-1))^\ell}{\ell!}t^\ell
\right)
\frac{e^{-t}}{t}\,dt.
\end{aligned}
\]

\medskip

Applying the Frullani identity with \(s=a+w-1\), the remainder integrals yield
\[
\begin{aligned}
\sum_{a=1}^{m}\chi(a)
\sum_{k=-1}^{n} m^k b_k
\frac{(-1)^{\,n-k+1}}{(n-k)!}
\Bigg(
&(a+w)^{n-k}\log(a+w)
\\
&-
\sum_{\ell=1}^{n-k}
\binom{n-k}{\ell}
\bigl(H_{n-k}-H_{n-k-\ell}\bigr)
(a+w-1)^\ell
\Bigg).
\end{aligned}
\]

\medskip

Reindexing via \(\ell = n-k\) and reorganizing the resulting sums, we finally obtain
\[
\begin{aligned}
&\int_0^{+\infty}
\frac{1}{t^n}
\left(
\frac{F_{\chi}(-t,w-1)}{t}
-
\sum_{k=-1}^{n}
m^k b_k t^k
\sum_{a=1}^{m} \chi(a)e^{-(a+w-1)t}
\right)
\frac{e^{-t}\,dt}{t}
\\
&=
\int_0^{+\infty}
\frac{1}{t^n}
\left(
\frac{F_{\chi}(-t,w-1)}{t}
-
\sum_{k=-1}^{n}
b_{\chi,k}(w)t^k
\right)
\frac{e^{-t}\,dt}{t}
\\
&\quad-
\sum_{\ell=0}^{n+1}
b_{n-\ell} m^{n-\ell}
\frac{(-1)^{\ell+1}}{\ell!}
\sum_{a=1}^{m}\chi(a)(a+w)^{\ell}\log(a+w)
\\
&\quad+
\sum_{\ell=0}^{n+1}
b_{n-\ell} m^{n-\ell}
\frac{(-1)^{\ell+1}}{\ell!}
\sum_{j=1}^{\ell}
\binom{\ell}{j}
\bigl(H_{\ell}-H_{\ell-j}\bigr)
\sum_{a=1}^{m}\chi(a)(a+w-1)^j.
\end{aligned}
\]

Reassembling all contributions, we deduce
\[
\begin{aligned}
\hat{S}^{(k+w)^n,(k+w)}(\chi)
&=
(-1)^{n+1}n!
\int_0^{+\infty}
\frac{1}{t^n}
\left(
\frac{F_{\chi}(-t,w-1)}{t}
-
\sum_{k=-1}^{n}
b_{\chi,k}(w)t^k
\right)
\frac{e^{-t}\,dt}{t}
\\
&\quad+
(-1)^{n+1}n!
\sum_{\ell=0}^{n+1}
b_{n-\ell} m^{n-\ell}
\frac{(-1)^{\ell+1}}{\ell!}
\sum_{j=1}^{\ell}
\binom{\ell}{j}
\bigl(H_{\ell}-H_{\ell-j}\bigr)
\sum_{a=1}^{m}\chi(a)(a+w-1)^j
\\
&\quad+
(-1)^{n+1}n!
\sum_{\ell=0}^{n+1}
b_{n-\ell} m^{n-\ell}
\frac{(-1)^{\ell}}{\ell!}
\sum_{a=1}^{m}\chi(a)(a+w)^{\ell}H_{\ell}
\\
&\quad+
(-1)^{n+1}n!\, b_{\chi,n}(w+1)\log m.
\end{aligned}
\]

This concludes the proof upon combining \eqref{eq:105_1} and \eqref{eq:145_1} in the final identification step.

\medskip

It remains to simplify the term
\[
(-1)^{n+1}n!
\sum_{\ell=0}^{n+1}
m^{n-\ell} b_{n-\ell}
\frac{(-1)^{\ell+1}}{\ell!}
\sum_{j=0}^{\ell}
\binom{\ell}{j}
\bigl(H_{\ell}-H_{\ell-j}\bigr)
\sum_{a=1}^{m}\chi(a)(a+w-1)^j.
\]

From \eqref{eq:145_2}, we recall that
\[
\begin{aligned}
\sum_{a=1}^m\chi(a)
\sum_{\ell=0}^{n+1}
m^{n-\ell} b_{n-\ell}
\frac{(-1)^{\ell+1}}{\ell!}
\sum_{j=0}^{\ell}
\binom{\ell}{j}
H_{\ell-j}(a+w-1)^j
=
\sum_{j=0}^{n+1}
\frac{(-1)^{j+1}}{j!}H_j\, b_{\chi,n-j}(w)\,.
\end{aligned}
\]

Consequently, we obtain
\[
\begin{aligned}
&(-1)^{n+1}n!
\sum_{\ell=0}^{n+1}
m^{n-\ell} b_{n-\ell}
\frac{(-1)^{\ell+1}}{\ell!}
\sum_{j=0}^{\ell}
\binom{\ell}{j}
\bigl(H_{\ell}-H_{\ell-j}\bigr)
\sum_{a=1}^{m}\chi(a)(a+w-1)^j
\\
&=
(-1)^{n+1}n!
\sum_{a=1}^{m}\chi(a)
\sum_{\ell=0}^{n+1}
m^{n-\ell} b_{n-\ell}
\frac{(-1)^{\ell+1}}{\ell!}
H_{\ell}(a+w)^{\ell}
\\
&\quad-
(-1)^{n+1}n!
\sum_{j=0}^{n+1}
\frac{(-1)^{j+1}}{j!}H_j\, b_{\chi,n-j}(w)\,.
\end{aligned}
\]

We finally obtain the following expression for
\(\hat{S}^{(k+w)^n,(k+w)}(\chi)\):
\[
\begin{aligned}
\hat{S}^{(k+w)^n,(k+w)}(\chi)
&=
(-1)^{n+1}n!
\int_0^{+\infty}
\frac{1}{t^n}
\left(
\frac{F_{\chi}(-t,w-1)}{t}
-
\sum_{k=-1}^{n}
b_{\chi,k}(w)t^k
\right)
\frac{e^{-t}\,dt}{t}
\\
&\quad-
(-1)^{n+1}n!
\sum_{j=0}^{n+1}
\frac{(-1)^{j+1}}{j!}H_j\, b_{\chi,n-j}(w)
\\
&\quad+
(-1)^{n+1}n!\, b_{\chi,n}(w+1)\log m\,.
\end{aligned}
\]

\end{proof}

\subsubsection{Asymptotics for twisted sums}

We study the asymptotic expansion of twisted partial sums over arithmetic progressions. For $\ell \in \{0,1,\dots,m-1\}$, we consider
\begin{equation}\label{eq:asymp_ell_clean}
\sum_{k=1}^{ms+\ell} \chi(k)\, k^n \log k
=
\hat{A}_{\chi,n,\ell}(s)\log s
+
\hat{B}_{\chi,n,\ell}(s)
+
\hat{S}_{n,\ell}(\chi)
+
\hat{R}_{\chi,n,\ell}\!\left(\frac{1}{s}\right).
\end{equation}

In particular, we have
\[
\hat{S}_{n,0}(\chi)=\hat{S}^{k^n,k}(\chi).
\]

For $\ell=1,\dots,m-1$, a direct decomposition along residue classes yields an expansion of the form
\[
\sum_{k=1}^{\ell} \chi(k)\,(ms+k)^n \log(ms+k)
=
\hat{a}_{\chi,n,\ell}(s)\log s
+
\hat{b}_{\chi,n,\ell}(s)
+
\hat{s}_{n,\ell}(\chi)
+
\hat{r}_{\chi,n,\ell}\!\left(\frac{1}{s}\right),
\]
where the coefficients satisfy:
\begin{itemize}
\item $\hat{a}_{\chi,n,\ell}(s) \in \mathbb{Q}(\chi)[s]$,
\item $\hat{b}_{\chi,n,\ell}(s) \in s\big(\mathbb{Q}(\chi)\oplus \log m\,\mathbb{Q}(\chi)\big)[s]$,
\item $\hat{s}_{n,\ell}(\chi) \in \mathbb{Q}(\chi)\oplus \log m\,\mathbb{Q}(\chi)$,
\item $\hat{r}_{\chi,n,\ell}(1/s) \in \mathbb{Q}(\chi)[[1/s]]$.
\end{itemize}

These coefficients can be made fully explicit. For $n\geq 1$, the following relations hold:
\begin{align*}
\hat{A}_{\chi,n,\ell}(s) &= \hat{A}_{\chi,n,0}(s) + \hat{a}_{\chi,n,\ell}(s), \\
\hat{B}_{\chi,n,\ell}(s) &= \hat{B}_{\chi,n,0}(s) + \hat{b}_{\chi,n,\ell}(s) \pmod{s\mathbb{Q}(\chi)[s]}, \\
\hat{S}_{n,\ell}(\chi) &= \hat{S}_{n,0}(\chi) + \hat{s}_{n,\ell}(\chi), \\
\hat{R}_{\chi,n,\ell}(1/s) &= \hat{R}_{\chi,n,0}(1/s) + \hat{r}_{\chi,n,\ell}(1/s).
\end{align*}

Moreover, the shift of the constant term is given explicitly by
\begin{equation}\label{eq:shift_constant_clean}
\hat{s}_{n,\ell}(\chi)
=
\bigl(H_n + \log m\bigr)
\sum_{k=1}^{\ell} \chi(k)\, k^n,
\end{equation}
where $H_n$ denotes the $n$-th harmonic number. In particular, it suffices to treat the case $\ell=0$.

\medskip

We now focus on $\hat{S}_{n,0}(\chi)$. Using Theorem \ref{thm:105_2}, we obtain Theorem \ref{thm:twisted_formula}.

\begin{proof}[Proof of Theorem \ref{cor:205_1}]
Recall that the twisted Stirling--Ramanujan constant is given by
\[
\hat{S}_{n,0}(\chi)
=
m^n \sum_{a=0}^{m-1} \chi(a)\, \mathbf{S}_n\!\left(\frac{a}{m}\right)\,.
\]

Moreover, the Dirichlet $L$-function admits the representation
\[
L_\chi(s)
=
\frac{1}{m^s}
\sum_{a=1}^{m} \chi(a)\,
\zeta_H\!\left(s,\frac{a}{m}\right)\,.
\]

Differentiating at $s=-n$ yields
\[
L_\chi'(-n)
=
m^n \sum_{a=1}^{m} \chi(a)\,
\zeta_H'\!\left(-n,\frac{a}{m}\right)
+
\frac{B_{\chi,n+1}}{n+1}\,  \log m,
\]
where we have used the identity $L_\chi(-n)=-\frac{B_{\chi,n+1}}{n+1}$. 

Using Theorem \ref{thm:zetasn}, we have
\[
\mathbf{S}_n\!\left(\frac{a}{m}\right)
=
-\zeta_H'\!\left(-n,\frac{a}{m}\right)
+
(-1)^{n+1}n!\, b_n\!\left(\frac{a}{m}\right) H_n,
\quad a=1,\ldots,m.
\]

Substituting this identity into the previous expression gives
\[
\begin{aligned}
L_\chi'(-n)
&=
m^n \sum_{a=1}^{m} \chi(a)
\left(
-\mathbf{S}_n\!\left(\frac{a}{m}\right)
+
(-1)^{n+1}n!\, b_n\!\left(\frac{a}{m}\right) H_n
\right)
\\
&\quad+
\frac{B_{\chi,n+1}}{n+1}\,  \log m,
\end{aligned}
\]
which, by the definition of $\hat{S}_{n,0}(\chi)$, yields the stated formula.
\end{proof}

\medskip

We are now in a position to prove Theorem \ref{thm:dirichletperiod}.
Let $\chi$ be an imprimitive character modulo $m$. Then there
exist a proper divisor $m_1$ of $m$ and a primitive character
$\chi_1$ modulo $m_1$ that induces $\chi$, namely,
\[
\chi(n)=
\begin{cases}
\chi_1(n) & \text{if } (n,m)=1,\\
0 & \text{if } (n,m)>1.
\end{cases}
\]

This relation implies a simple relation between the corresponding
$L$-functions. By the Euler product formula,
\begin{equation}\label{eq:LL}
\begin{aligned}
L_\chi(s)
&=
L_{\chi_1}(s)
\prod_{\substack{p\mid m\\p\nmid m_1}}
\left(1-\chi_1(p)p^{-s}\right).
\end{aligned}
\end{equation}

The primitive $L$-function $L_{\chi_1}(s)$ satisfies the functional
equation. Let
\[
\xi(s,\chi_1)
=
\left(\frac{\pi}{m_1}\right)^{-\frac{s+a}{2}}
\Gamma\left(\frac{s+a}{2}\right)
L_{\chi_1}(s),
\]
where $a$ is determined by
\[
a=
\begin{cases}
0 & \text{if } \chi_1(-1)=1,\\
1 & \text{if } \chi_1(-1)=-1.
\end{cases}
\]
Then the functional equation takes the form
\begin{equation}
\xi(1-s,\overline{\chi}_1)
=
\frac{i^a m_1^{1/2}}{\tau(\chi_1)}
\xi(s,\chi_1),
\tag{14}
\end{equation}
where
\[
\tau(\chi_1)
=
\sum_{k=1}^{m_1}
\chi_1(k)e^{2\pi i k/m_1}
\]
is the Gauss sum associated with $\chi_1$.

Finally, using \eqref{eq:LL}
together with the functional equation for $L_{\chi_1}(s)$, we obtain,
for $n\geq 1$ with $n\equiv a\pmod{2}$,
\[
\frac{L_{\overline \chi}(n+1)}{\pi^n}
=
(-1)^{(n-a)/2}
\frac{i^a 2^{n+1}}
{\tau(\chi_1)m_1^n n!}
\prod_{\substack{p\mid m\\p\nmid m_1}}
\frac{1-\overline\chi_1(p)p^{-n-1}}
     {1-\chi_1(p)p^n}
L_{\chi}'(-n)\;.
\]

%Let $\chi$ be a Dirichlet character modulo $m$, induced by the primitive character $\chi_1$ of conductor $m_1$, and let \[\chi(-1)=(-1)^a. \] For $n\geq 1$ with $n\equiv a\pmod{2}$, we have \[\frac{L_{\overline \chi}(n+1)}{\pi^n} = (-1)^{(n-a)/2} \frac{i^a 2^{n+1}} {\tau(\chi_1)m_1^n n!} \prod_{\substack{p\mid m\\ p\nmid m_1}} \frac{1-\overline\chi_1(p)p^{-n-1}}     {1-\chi_1(p)p^n} L_\chi'(-n)\;.\]

This proves  Theorem \ref{thm:dirichletperiod}.

\section{Appendix:  $\chi$-Bernoulli numbers and polynomials.}\label{appendix:Bernoulli}
In this Appendix we record  some basic formulas related to $\chi$-Bernoulli numbers and polynomials.

\subsection{Definition of classical Bernoulli Numbers}

Let $t$ be a variable and
\[
F(t) = \frac{te^t}{e^t - 1} \ .
\]

Consider the formal power series  of  $F(t)$,
\[
F(t) = \sum_{n=0}^{\infty} B_n \frac{t^n}{n!} \ .
\]
The coefficients $B_n$, $n \geq 0$, are rational and are the \emph{Bernoulli numbers}. Sometimes $t^{-1}F(t)$ is used as a generating function instead of $F(t)$. The first ones are:
\[
B_0 = 1, \quad B_1 = \frac{1}{2}, \quad B_2 = \frac{1}{6}, \quad B_3 = 0, \quad \ldots
\]

Since $F(-t) = F(t) - t$, we have that $B_n = 0$ for odd  $n \geq 3$.

\subsection{Periodic functions and Dirichlet characters}

 We consider $\chi : \Z\to \C$ to be an arbitrary periodic function with period $m\geq 1$. 
 A particular case occurs when $\chi$ is a Dirichlet character $\chi$ modulo $n$. Then the minimal period $m$ divides $n$, and $m=n$ when
$\chi$ is a primitive character. For Dirichlet characters $\chi$, Leopoldt \cite{Leo} defined $\chi$-Bernoulli numbers and polynomials. We have formulas that generalize the classical ones (see for example \cite{Iwa}) and generalizations of classical congruence properties of the type of Staudt-Clausen (see \cite{Leo} and \cite{Car}). We have similar formulas in the general case when $\chi$ is only assumed to be a periodic function. We refer to \cite{Co} for the general approach of defining $\chi$-Bernoulli numbers and polynomials in this more general setup that seems better suited for twisted sums.

 Note that in the case of a Dirichlet character, we have the multiplicative property, for all $a,b\in \Z$, $\chi(ab)=\chi(a).\chi(b)$, and $\chi$ induces a character of the group $(\Z/m\Z)^*$ into $\C^*$. The multiplicative property implies that $\chi(0)^2=\chi(0)$, so $\chi(0)=0$ or $\chi(0)=1$. In this last case, for all $a\in \Z$,  $\chi(a)=\chi(a).\chi(0)=\chi(0)=1$, so $\chi =\chi^0$ is the constant function equal to $1$,   and we have period $m=1$. So, when the minimal period $m\geq 2$, we have  $\chi(0)=0$. In other words, except in the case when $\chi=\chi^0$ 
  , we always have $\chi(0)=0$ for a Dirichlet character. For a Dirichlet character we also have that the average value over a period vanishes,
 \[
\sum_{a=1}^m \chi(a) =0
 \]
Obviously the set of periodic functions $\chi$ is a $\Z$-module but the subset of Dirichlet characters is not. 
\subsection{$\chi$-Bernoulli Numbers}

\begin{definition}
Let $\chi:\Z\to \C$ be a periodic function of period $m\geq 1$. The \emph{$\chi$-Bernoulli numbers} ${B_{n,\chi}}$ and \emph{polynomials} ${B_{n,\chi}}(x)$ associated with $\chi$ are defined by the generating functions:
\[
F_{\chi,m}(t) = \sum_{a=1}^{m} \frac{\chi(a)te^{at}}{e^{mt} - 1} = \sum_{n=0}^{\infty} {B_{n,\chi, m}} \frac{t^n}{n!}, \quad F_{\chi,m}(t, x) = F_{\chi,m}(t)e^{xt} = \sum_{n=0}^{\infty} {B_{n,\chi, m}}(x) \frac{t^n}{n!}  \ .
\]
\end{definition}
We can drop the subscript $m$ and assume that $m$ is the minimal period of $\chi$ because we have for $k\geq 1$,
$F_{\chi,km}=F_{\chi,m}$.
\begin{align*}
F_{\chi,km}(t) &=\sum_{a=1}^{km} \frac{\chi(a)te^{at}}{e^{kmt} - 1} = \sum_{\ell=0}^{k-1}
\sum_{a=\ell m +1}^{(\ell+1)m} \frac{\chi(a)te^{at}}{e^{kmt} - 1} \\
&=\sum_{\ell=0}^{k-1}
\sum_{a=1}^{m} \frac{\chi(a)te^{(a+\ell m)t}}{e^{kmt} - 1} \\
&=\left (\sum_{\ell=0}^{k-1} e^{\ell mt}\right ) \sum_{a=1}^{m} \frac{\chi(a)te^{at}}{e^{kmt} - 1} \\
&=\frac{e^{kmt}-1}{e^{mt}-1} \sum_{a=1}^{m} \frac{\chi(a)te^{at}}{e^{mt} - 1}\\
&=\sum_{a=1}^{m} \frac{\chi(a)te^{at}}{e^{mt} - 1} \\
&=F_{\chi,m}(t)
\end{align*}

We use the notation $\chi^-(n)=\chi(-n)$. If $\chi$ is a Dirichlet character then $\chi^- =\chi(-1) \chi$, and
\[
F_{\chi^-}(t) =\chi(-1) F_{\chi}(t) \ \ , \  F_{\chi^-}(t,x) =\chi(-1) F_{\chi}(t,x)
\]
and
\[
B_{\chi^-} =\chi(-1) B_{\chi} \ \ , \  B_{\chi^-}(x) =\chi(-1) B_{\chi}(x) \ .
\]

\begin{proposition}[Properties of the $\chi$-Bernoulli numbers and polynomials]\label{prop:Bernoulli}
We have the following properties:
\medskip

\begin{enumerate}
    \item For 
     $\chi$ the constant function equal to $1$ and $m=1$, we recover  
     the classical Bernoulli numbers $B_n=B_{\chi,1}$ and polynomials $B_n(x)=B_{\chi,1}$.
    \item We have the binomial identity ${B_{n,\chi}}(x) = \sum_{j=0}^{n} \binom{n}{j} {B_{j,\chi}}x^{n-j}$. In particular, ${B_{n,\chi}}(0) = {B_{n,\chi}}$.

    \item When $\chi(0)=0$, the symmetry of the generating functions
    \begin{equation}\label{eq_2}
    F_{\chi}(-t)=F_{\chi^-}(t) \ \ ,\ F_{\chi}(-t, -x) = F_{\chi^-}(t, x)
    \end{equation}
    yields the parity identity $(-1)^n {B_{n,\chi}}(-x) = {B_{n,\chi^-}}(x)$ for $n \geq 0$.
    \item We have the  identity
    \begin{equation*}
    F_{\chi}(t,x)-F_{\chi}(t,x-m) =t\sum_{a=1}^m \chi(a) e^{(a+x-m)t}
    \end{equation*}
    so
    \begin{equation}\label{eq_1}
    {B_{n,\chi}}(x) - {B_{n,\chi}}(x-m) = n \sum_{a=1}^{m} \chi(a) (a+x-m)^{n-1}
    \end{equation}
    \item For $k\geq 1$, let  $S_{n,\chi}(k) = \sum_{a=1}^{k} \chi(a) a^n$. We have for $n, k \geq 0$:
    \[
    S_{n,\chi}(km) = \frac{1}{n+1} \left( {B_{n+1,\chi}}(km) - {B_{n+1,\chi}} \right).
    \]
    In particular, for $\chi$ the constant function $1$, $m=1$, this yields Faulhaber's formula $S_n(k) = \frac{1}{n+1} ({B_{n+1}}(k) - {B_{n+1}})$.
If $\frac{1}{m} \sum_{a=1}^{m} \chi(a) = 0$, then we have $B_{0,\chi,m} = \frac{1}{m}\sum_{a=1}^{m} \chi(a) = 0$, and $\deg {B_{n,\chi}} < n$ for $n \geq 1$.
\end{enumerate}
\end{proposition}

\begin{proof}
These are straightforward computations. For point (3), we have,
\begin{align*}
F_\chi(-t) &= \sum_{a=1}^m \frac{\chi(a) (-t) e^{-at}}{e^{-mt}-1} =\sum_{a=1}^m \frac{\chi(a) t e^{-at}}{e^{mt}-1} \, e^{mt}=\sum_{a=1}^m \frac{\chi(a) t e^{(m-a)t}}{e^{mt}-1} \\
&=\sum_{a=1}^m \frac{\chi^-(m-a) t e^{(m-a)t}}{e^{mt}-1} =\sum_{a=0}^{m-1} \frac{\chi^-(a) t e^{at}}{e^{mt}-1} =\sum_{a=1}^{m} \frac{\chi^-(a) t e^{at}}{e^{mt}-1} \\
&=F_{\chi^-}(t)
\end{align*}

For point (5) we add the identities (\ref{eq_1}) for $x=m, 2m, \ldots ,km$.
\end{proof}

\end{document}